\documentclass[
12pt,
fleqn]
{article}
\usepackage[english]{babel}
\usepackage[T1]{fontenc}
\usepackage{amsmath}
\usepackage{xcolor}
\usepackage{hyperref}
\usepackage{amssymb}
\usepackage{amsthm}

\usepackage[title]{appendix}

\newtheorem{theorem}{Theorem}[section]
\newtheorem{proposition}[theorem]{Proposition}
\newtheorem{lemma}[theorem]{Lemma}

\theoremstyle{definition}

\newtheorem{remark}[theorem]{Remark}

\usepackage{graphicx}
\DeclareMathOperator{\dv}{div}

\DeclareMathOperator{\tr}{tr}
\DeclareMathOperator{\diam}{diam}
\DeclareMathOperator{\dist}{dist}

\DeclareMathOperator{\interior}{int}

\DeclareMathOperator{\cl}{clos}

\newcommand{\R}{{\mathbb R}}
\newcommand{\C}{{\mathbb C}}

\newcommand{\e}{\epsilon}
\DeclareMathOperator{\im}{\mathfrak{Im}}
\DeclareMathOperator{\re}{\mathfrak{Re}}

\title{On $C^1$ regularity for  degenerate elliptic equations in the plane}

\author{Thibault Lacombe\footnote{Institut de Math\'ematiques de Toulouse, UMR 5219, Universit\'e de Toulouse, CNRS, UPS
IMT, 31062 Toulouse Cedex 9, France. \tt{thibault.lacombe@math.univ-toulouse.fr}}
 \and 
Xavier Lamy\footnote{Institut de Math\'ematiques de Toulouse, UMR 5219, Universit\'e de Toulouse, CNRS, UPS
IMT, 31062 Toulouse Cedex 9, France. \tt{xavier.lamy@math.univ-toulouse.fr}}
}

\begin{document}

\maketitle

\begin{abstract}
We show that Lipschitz solutions 
$u$ of 
$\mathrm{div}\, G(\nabla u)=0$ 
in $B_1\subset\R^2$
are $C^1$, 
for strictly monotone vector fields $G\in C^0(\mathbb R^2;\mathbb R^2)$
satisfying a mild ellipticity condition.
If $G=\nabla F$ for a strictly convex function $F$,
and $0\leq \lambda(\xi)\leq \Lambda(\xi)$ 
are
 the two eigenvalues of $\nabla^2 F(\xi)$,
our 
 assumption
 is that the set
$\lbrace\lambda=0\rbrace
\cap \lbrace \Lambda=\infty\rbrace$,
 where ellipticity degenerates
 \emph{both} from below and from above,
is finite.
This extends results by De Silva and Savin
(Duke Math. J. 151, No. 3, p.487-532, 2010), 
which assumed either that set empty, or the larger
 set $\lbrace \lambda=0\rbrace$  finite.
Our main new input is to
transfer estimates in 
$\lbrace \lambda > 0 \rbrace $ to estimates in $\lbrace \Lambda <\infty\rbrace$
 by means of a conjugate equation.
When $G$ is not a gradient, 
the ellipticity assumption needs to be interpreted in a specific way, 
and we highlight the nontrivial effect of the antisymmetric part of $\nabla G$.
\end{abstract}

\section{Introduction}

We consider solutions $u\colon B_1\subset\R^2\to \R$ of the 
nonlinear
 equation
\begin{equation}
\label{2i}
\dv G(\nabla u) = 0
\qquad
 \text{ in } B_1,
\end{equation}  
where the vector field $G : \mathbb{R}^2 \to \mathbb{R}^2$ is strictly monotone, that is,
\[ 
\langle G(\xi) - G(\zeta), \xi - \zeta \rangle > 0 \quad \forall \xi \neq \zeta \in \mathbb{R}^2.
\] 
Sufficient conditions on $G$ ensuring local Lipschitz regularity of weak solutions
 can be found in \cite{EMT} 
(see also the survey \cite{MingioneDarkSide}, 
 the recent works \cite{BBJ16,BB16,BBLV18,BB20,BBL24},
and counterexamples in \cite{CT24,ACFJKM} when integrability properties of $\nabla u$ are not good enough).
Here, as in  \cite{DSS,BB18,LR18,CMCE}, we focus on conditions
 ensuring that Lipschitz solutions are $C^1$.

 It is known since the works of Morrey and Nirenberg \cite{morrey38,nirenberg53} 
 (see \cite[Chapter~12]{GT}) 
that Lipschitz solutions of \eqref{2i} are 
$C^{1,\alpha}$ for some $\alpha\in (0,1)$
 (and smooth if $G$ is smooth) if $G$ is \emph{uniformly elliptic}, that is,
 \begin{align*}
 \lambda\leq \nabla^s G \leq \Lambda
 \qquad\text{for some }0 < \lambda < \Lambda < +\infty\, ,
 \end{align*} 
where $\nabla^s G=(\nabla G +\nabla G^T)/2$ denotes the symmetric gradient of $G$.

In this two-dimensional setting, 
the antisymmetric part of $\nabla G$ does not play any role in the uniform ellipticity condition, 
because the regularity can be inferred from a priori estimate 
for the
equation
%in non-divergence form 
$\tr(A \nabla^2 u)=0$, 
with $A=\nabla^s G(\nabla u)$.
In higher dimensions, the estimates of de Giorgi, Nash and Moser 
also provide $C^{1,\alpha}$ regularity for \eqref{2i}
 under the uniform ellipticity assumption
 $\lambda \leq \nabla^s G \leq |\nabla G|\leq \Lambda$, 
 which  does  require the antisymmetric part of $\nabla G$ to be bounded.

For a general strictly monotone $G$, 
%we only know a priori that $\nabla^s G\geq 0$, 
%and 
ellipticity may degenerate from below: $\lambda = 0$; or   above: $\Lambda = +\infty$. 
Specific types of degeneracy have been studied extensively, including the fundamental case of the $p$-Laplacian $G=\nabla |\cdot|^p$ for $1<p<\infty$, 
% \cite{UR,UL,Evans,JLLEWIS,Tolksdorff}, 
but the general setting still raises several open questions, see \cite{CM} for  a recent survey.
Very degenerate equations, where the field $G$ is not strictly monotone, have also attracted recent  attention \cite{SV,CF14a,CF14b,Lledos}.

%We start with a non-rigorous description of our results.

\subsection{The variational case $G=\nabla F$}

One
feature of the present work
is to highlight the nontrivial effect of the antisymmetric part of $\nabla G$, 
which played no role in the
two-dimensional
 uniformly elliptic case. 

We focus first on the case where this effect is absent, that is, $\nabla G$ is symmetric.
Then we have $G=\nabla F$ for some strictly convex function $F$. Lipschitz solutions of
\begin{equation}
\label{eq3i}
\dv \nabla F(\nabla u)  = 0 \text{ in } B_1\,,
\end{equation} 
 are minimizers of the energy $\int F(\nabla u)\, dx$, hence the term \emph{variational}.

Following \cite{DSS}, it is natural to separate regions where ellipticity degenerates: $\lambda=0$, or becomes singular: $\Lambda=+\infty$. 
Accordingly we set,
leaving regularity considerations and rigorous definitions 
aside for the moment,
\begin{align*}
\mathcal D
&
=
\left\lbrace
\xi\in\R^2\colon
\text{`` } \nabla^2 F (\xi) \text{ has an eigenvalue equal to }0 
\text{ ''}
\right\rbrace,
\\
\mathcal S
&
=
\left\lbrace
\xi\in\R^2\colon
\text{`` } \nabla^2 F (\xi) \text{ has an eigenvalue equal to } +\infty
\text{ ''} \right\rbrace.
\end{align*}

%Here $\cl A$ denotes the topological closure of a set $A\subset\R^2$.
In \cite{DSS}, 
de Silva and Savin show that $M$-Lipschitz solutions of \eqref{eq3i},
with $F$ strictly convex,
 are $C^1$, provided one of the two following conditions is satisfied:
\begin{itemize}
\item $\overline B_M\cap \mathcal D\cap \mathcal S $ is empty: ellipticity is not lost simultaneously from below and above \cite[Theorem~1.1]{DSS};
\item $ \overline B_M\cap \mathcal D$ is finite: ellipticity is lost from below at most at a finite number of values \cite[Theorem~1.2]{DSS}.
\end{itemize}
A remarkable feature of \cite[Theorem~1.2]{DSS} is that it requires no control from above on the ellipticity.
It is natural to wonder about a counterpart involving only $\mathcal S$.
As a consequence of our main result we obtain 
$C^1$ regularity under
 an even less restrictive condition,
  generalizing
   both 
\cite[Theorem~1.1 \& 1.2]{DSS}.

\begin{theorem}\label{t:var}
If $F \colon \R^2 \to \R $ is $C^1$ and strictly convex, 
and 
\begin{align*}
\overline{B}_M \cap \mathcal{D} \cap \mathcal{S}\text{ is finite},
\end{align*}
 then any Lipschitz solution $u$ of $\dv \nabla F(\nabla u)=0$ in $B_1$
  with $|\nabla u|\leq M$ is $C^1$. 
\end{theorem}

We describe next a family of examples where Theorem~\ref{t:var} applies, but the results of \cite{DSS} do not.
Consider $F(x,y)=f(x)+g(y)$ where $f,g\colon\R\to\R$ are $C^1$ strictly convex functions.
Denote by $D_f,S_f\subset\R$ the closures of the sets where $f''=0$ or $+\infty$, and similarly for $g$.
Since $D_f\cap S_f=D_g\cap S_g=\emptyset$,
we have
\begin{align*}
\mathcal D
&
=(D_f\times\R )\cup (\R\times D_g),
\\
\mathcal S 
&
=(S_f\times \R)\cup (\R\times S_g),
\\
\mathcal D\cap\mathcal S 
&
=(D_f\times S_g)\cup (S_f\times D_g).
\end{align*}
Hence $\mathcal D$ is infinite as soon as $D_f$ or $D_g$ is non-empty, 
but the intersection $\mathcal D\cap\mathcal S$ is finite and non-empty as soon as $D_f\cup S_f$ and $D_g\cup S_g$ are finite and non-empty.

Note that, in the case of the $p$-Laplacian, ellipticity is degenerate only at the origin: bad values of the gradient are small, which facilitates regularity theory.
Theorem~\ref{t:var} goes well beyond that case, as did already the results of \cite{DSS}.
However, as in \cite{DSS}, our approach is purely two-dimensional.
A counter-example in dimension 4 is given in \cite{CMCE}, 
where it is also conjectured that there should be a counter-example in dimension 3.

\subsection{The general case}

Consider a general $C^0$ strictly monotone vector field $G\colon\R^2\to\R^2$.
In the
 uniformly elliptic case, 
the condition satisfied by $G$ 
constrains only the symmetric part $\nabla^s G$ of the gradient, 
hence a natural generalization of the sets 
$\mathcal D$ and $\mathcal S$ 
 could
 be 
the sets where $\nabla^s G$ has a zero or infinite eigenvalue.
This turns out to be 
wrong for $\mathcal S$, as we explain next.

Note that $G$ is strictly monotone,
 and therefore injective,
  so we may consider its inverse $G^{-1}$.
Then, the (loose) definitions
\begin{align*}
\mathcal D(G)
&
 =
 \left\lbrace \xi\in\R^2\colon
\text{`` } \nabla^s G(\xi)\text{ has an eigenvalue equal to }0
\text{ ''}\right\rbrace\, ,
\\
\mathcal S(G)
&
= 
\left\lbrace \xi\in\R^2\colon 
\text{`` }\nabla^s (G^{-1})(G(\xi))
\text{ has an eigenvalue equal to }0
\text{ ''}\right\rbrace\, ,
\end{align*}
ensure the validity of an exact analog of Theorem~\ref{t:var} in the general strictly monotone setting. 
Before stating it, we give the rigorous versions:
\begin{align}\label{eq:DS}
\begin{aligned}
\mathcal{D}(G) & =
\bigcap_{\lambda>0} \cl 
\left\lbrace \xi \in \mathbb{R}^2 : \liminf_{\vert  \zeta \vert \to 0} \frac{\langle G(\xi +\zeta) - G(\xi), \zeta \rangle}{\vert \zeta \vert^2} \leq\lambda \right\rbrace 
\\
\mathcal{S}(G) & =\bigcap_{\Lambda>0} \cl 
\left\lbrace \xi \in \mathbb{R}^2 : \liminf_{\vert \zeta \vert \to 0} \frac{\left\langle G( \xi +\zeta) - G(\xi), \zeta \right\rangle}{\left\vert G(\xi+ \zeta) - G(\xi) \right\vert^2} \leq
\frac 1\Lambda
\right\rbrace.
\end{aligned}
\end{align}
Here, $\cl A$ denotes the topological closure of $A\subset\R^2$.

\begin{remark}\label{r:defSD}
These definitions make sense without any regularity assumption on $G$.
The quotient appearing in the definition of $\mathcal S(G)$ can be rewritten as
\begin{align*}
\frac{\langle G^{-1}(G(\xi) +\eta) - G^{-1}(G(\xi)), \eta \rangle}{\vert \eta \vert^2},
\quad \text{with }\eta =G(\xi+ \zeta) - G(\xi),
\end{align*}
which explains the previous 
loose
 definition. 
If $G$ is $C^1$, then $\mathcal D(G)$ is the set where $\nabla G$ has a zero eigenvalue, and similarly for $\mathcal S(G)$ if $G^{-1}$ is $C^1$. 
In general, such a pointwise description fails.
For instance,
the inclusion
\begin{align*}
D(G):=\cl \left\lbrace\xi\in\R^2 \colon 
\liminf_{\zeta \to 0} 
\frac{\langle G(\xi+\zeta)-G(\xi),\zeta \rangle}{\vert \zeta \vert^2 } = 0
\right\rbrace
 \subset \mathcal D(G),
\end{align*}
might be strict: consider  $G(x,y)=(g(x),y)$ with $g(x)=\int_0^x f$, where 
 $f(t)=|t|+|\sin(1/t)|$ for $t\neq 0$, then $D(G)=\emptyset$ but $\mathcal D(G)=\lbrace 0\rbrace$.
 \end{remark}

With these definitions, our main result is as follows.

\begin{theorem}\label{t:gen}
If $G \colon \R^2 \to \R^2 $ is $C^0$ and strictly monotone, 
and 
\begin{align*}
\overline{B}_M \cap \mathcal{D}(G) \cap \mathcal{S}(G)\text{ is finite},
\end{align*}
 then any Lipschitz solution $u$ of 
$\dv G(\nabla u)=0$ in $B_1$  with $|\nabla u|\leq M$
  is $C^1$. 
\end{theorem}

Note that, in the uniformly elliptic case $0<\lambda \leq \nabla^s G\leq \Lambda$ we have $\mathcal D(G)=\emptyset$,
so in that case Theorem~\ref{t:gen} also requires no condition on the antisymmetric part of $\nabla G$.
Next we give some further explanations about the role of that antisymmetric part and the set $\mathcal S(G)$.

\subsection{The role of the antisymmetric part of $\nabla G$}

In view of the uniformly elliptic case, the set
\begin{align*}
\widetilde{\mathcal S}(G)= \left\lbrace \xi\in\R^2
\colon \text{`` }\nabla^s G(\xi) \text{ has an eigenvalue equal to }+\infty\text{ ''}\right\rbrace\, ,
\end{align*}
or rather its rigorous version
\begin{align}\label{eq:tildeS}
\tilde{\mathcal{S}}(G) & = 
\bigcap_{\Lambda>0}
\cl \left\lbrace \xi \in \mathbb{R}^2 : \limsup_{\vert \zeta \vert \to 0} \frac{\left\langle G(\xi + \zeta) - G(\xi), \zeta \right\rangle}{\left\vert \zeta \right\vert^2} \geq\Lambda \right\rbrace
\,,
\end{align}
could have been a natural candidate to replace $\mathcal S$ in a 
 generalization of Theorem~\ref{t:var} to the nonvariational setting. 
We 
demonstrate instead
%
%why our definition in \eqref{eq:DS} 
%seems instead to be the correct one:
that the antisymmetric part of $\nabla G$ actually plays an important role in this degenerate setting. 
Before doing so, 
we observe the following elementary property
%We start by gathering a few elementary properties
 to help compare $\mathcal S$ and $\widetilde{\mathcal S}$.

\begin{proposition}\label{p:StildeS}
For any continuous strictly monotone $G \colon \R^2 \to \R^2 $, we have
the inclusion
$
\widetilde{\mathcal S}(G)
\subset
{\mathcal S}(G)
$,
with equality if $G=\nabla F$.
\end{proposition}

This 
implies
 in particular 
that Theorem~\ref{t:gen} reduces exactly to Theorem~\ref{t:var} when $G=\nabla F$.

Note that the inclusion $\widetilde{\mathcal S}(G)\subset \mathcal S(G)$ can be strict.
For instance, the strictly monotone field
 $G_0(x,y)=(x^3-y,x+y)$
 satisfies
  $\mathcal S(G_0) =\R\times \lbrace 0\rbrace $ and $\widetilde {\mathcal S}(G_0)= \emptyset $.
  This elementary example 
   is however not so interesting from the standpoint of equation \eqref{2i}, because its antisymmetric contribution is linear, and therefore disappears in \eqref{2i}.

Note also that the set $\mathcal S(G)$ can be non-empty even when $G$ is uniformly elliptic:
for any non-zero antisymmetric matrix $M$ of operator norm $\|M\|\leq 1/2$,
the field $G_M(\xi)=\xi +\ln(|\xi|)M\xi$ satisfies $1/2\leq \nabla^s G_M\leq 3/2$,
but  the antisymmetric part of $\nabla G_M $ blows up at the origin, and $\mathcal S(G_M)=\lbrace 0\rbrace$.

We  provide here 
a more involved example which shows that Theorem~\ref{t:gen} 
is false with $\mathcal S(G)$ replaced by $\widetilde{\mathcal S}(G)$.

\begin{theorem}\label{t:ex}
There exists a $C^0$ strictly monotone vector field $G \colon \R^2 \to \R^2$, with $\mathcal{D}(G)\cap \mathcal{\tilde{S}}(G)$ finite, and a Lipschitz solution of \eqref{2i} which is not $C^1$.
\end{theorem}

This example arises from connections with the  Aviles-Giga energy and degenerate differential inclusions \cite{LP,LLP,LLP24}, which partially motivated the present work.
The bad set $\mathcal D\cap \mathcal S$ is infinite, and the non-$C^1$ solution $u$ satisfies $\nabla u\in \mathcal D\cap\mathcal S$ almost everywhere.
In view of this and of Theorem~\ref{t:gen}, 
it seems natural to conjecture that, 
for quite general strictly monotone vector fields $G$ and any Lipschitz solution $u$ of \eqref{2i}, 
the function $x\mapsto \dist(\nabla u(x),\mathcal D\cap\mathcal S)$ is continuous.
Results in that spirit are proved in \cite{SV,CF14a,CF14b,Lledos}
 in different contexts.

In fact, our proof of Theorem~\ref{t:gen} does imply this continuity, 
under the assumption that 
 the complement of 
 any small enough
  neighborhood of $\mathcal D\cap\mathcal S$ is connected 
 (cf Remark~\ref{r:connect}).
But
 in the example of Theorem~\ref{t:ex} the complement of $\mathcal D\cap\mathcal S$ is not connected: additional arguments would be needed to prove the above conjecture.

\subsection{Nonlinear Beltrami equations}

In two dimensions, Minty's correspondence 
between monotone and 1-Lipschitz vector fields
\cite{Minty} 
induces a correspondence between equations of the form \eqref{2i} and nonlinear autonomous
Beltrami equations
\begin{equation}
\label{Beltrami}
f_{\bar{z}} = H(f_z),
\end{equation}
where $f_z=(\partial_x f-i\partial_y f)/2$, $f_{\bar z}=(\partial_x f +i\partial_y f)/2$,
and $H\colon\C\to\C$ is \emph{strictly 1-Lipschitz}, that is,
\begin{equation*}
\vert H(\xi) - H(\zeta) \vert < \vert \xi - \zeta \vert\qquad\forall \xi\neq\zeta\in\C\, .
\end{equation*}
This connection is described and exploited for instance in \cite[Chapter~15]{IMGeom}, \cite[Chapter~16]{AIM}, or \cite{AstalaClop}.
In another instance of degenerate elliptic setting, Minty's correspondence 
also has applications to the regularity of Monge-Ampère equations in the plane \cite{ADPK11,DPGT}.

The nonlinear Beltrami equation \eqref{Beltrami} is
 uniformly
 elliptic if the function $H$ is $k$-Lipschitz for some $k\in (0,1)$, 
and 
the $1$-Lipschitz case allows for 
degenerate ellipticity.
In that setting, we have the following transposition of Theorem~\ref{t:gen}.

% Roughly speaking, for each PDE of the form \( \operatorname{div}(G(\nabla u)) = 0 \) with \( G \) strictly monotone, we can construct a map \( H : \mathbb{C} \to \mathbb{C} \) and an application \( f : \mathbb{C} \to \mathbb{C} \) such that \( \mathfrak{Re}(f) = u \) and \( f \) solves \( f_{\bar{z}} = H(f_z) \). Conversely, given a map \( H \) defined on \( \mathbb{C} \), one can construct (under suitable assumptions on \( H \)) a strictly monotone vector field \( G \) such that \( \operatorname{div}(G(u_{\bar{z}})) = 0 \) where \( u = \mathfrak{Re}(f) \) and \( f \) solves \eqref{Beltrami}. If the function \( H \) is \( k \)-Lipschitz for \( k < 1 \), regularity theory is largely investigated; see \cite{AIM} for an introduction, and \cite{AstalaClop}, \cite{AstalaKC}, \cite{TGDP} for more recent results. Cases where the nonlinearity \( H \) in \eqref{Beltrami} is only \( 1 \)-Lipschitz are naturally connected with degenerate elliptic PDEs in the plane, and Theorem \ref{Thfinipoint} provides a new regularity result in this direction. This result is a simple corollary but cannot be obtained only through \cite[Theorem~1.2]{DSS}.

\begin{theorem}
\label{t:beltrami}
Let \( H : \mathbb{C} \to \mathbb{C} \) be 
strictly 1-Lipschitz and define
\begin{align*}
&
\Gamma_\pm
=\bigcap_{\lambda>0} \cl \left\lbrace \xi\in \C \colon 
\liminf_{\zeta \to 0} \frac{1 - \vert L_H(\xi,\zeta)\vert^2}{\vert 1 \pm L_H(\xi,\zeta)\vert^2} \leq \lambda
\right\rbrace,
\\
&
\text{where }
L_H(\xi,\zeta)
=\frac{H(\xi+\zeta)-H(\xi)}{\overline{\zeta}}.
\end{align*}
If $\Gamma_+\cap\Gamma_-$ is locally finite, then any Lipschitz solution $f\colon B_1\to\C$ of \eqref{Beltrami} is $C^1$.
\end{theorem}
A complete version, 
with the explicit role of the Lipschitz constant $M$ in Theorem~\ref{t:gen}, 
will be given in Theorem~\ref{t:beltramiM}. 
Many other types of regularizing effects of 
nonlinear Beltrami equations  have been studied, for instance in 
 \cite{ACFJS12,ACFJK17,ACFJ19,AstalaClop,GMSuper,ACFJKM}.

A few comments about the bad set $\Gamma_+\cap \Gamma_-$ are in order. 
As expected, it contains points 
$\xi\in\C$
at which the local Lipschitz constant
\begin{align*}
\mathrm{Lip}(H;\xi) = \limsup_{\zeta\to 0} |L_H(\xi,\zeta)|
\end{align*}
 is equal to  1.
However it takes into account more precise information, not only about the modulus of 
finite differences
 of $H$,
but also about their angle.
For instance, if $H$ is differentiable at a point $\xi_0$,
with local Lipschitz constant equal to 1,
but
 its differential is the 
conjugation
 operator
$\zeta\mapsto\bar \zeta$ (or its opposite), 
 then $\xi_0\notin\Gamma_-$ (or $\xi_0\notin\Gamma_+$).
This angular effect, 
and the particular role played by the conjugation operator, 
was already observed in \cite{GSVM} for linear Beltrami equations with varying coefficient,
and studied further in the context of quasilinear elliptic equations of the form $\dv( A(z,u)\nabla u )=0$, see the survey \cite{GRSS}.

\begin{remark}\label{r:conjug}
A basic interpretation of the distinguished role of the conjugation operator is provided by the following elementary fact about degenerate linear Beltrami equations
\begin{align*}
f_{\bar z}= L[ f_z],
\end{align*}
where $L\colon \mathbb C\to\C$ is a $\R$-linear operator of operator norm $\|L\|=1$.
If $\pm L$ is the conjugation operator,
 then $\re f$ or $\im f$ must be constant, and this is the only operator of norm 1 with that property (see Proposition~\ref{p:linbeltrami}).
In other words, at the linear level, the conjugation operator has a smoothing effect on one of the two components of $f$. 
The
 regularization of both components in Theorem~\ref{t:beltrami} is due to further nonlinear effects, 
making use of the fact that $H$ is strictly 1-Lipschitz.
\end{remark}

\subsection{Ideas of the proof of Theorem~\ref{t:gen}}

The strategy combines the ideas of \cite{DSS} 
with a duality argument 
%presented in \cite[\S~16.4]{AIM},
 which allows to symmetrize 
the roles of $\mathcal D$ and $\mathcal S$.

The main tool in the proof of \cite[Theorem~1.2]{DSS} is a localization lemma 
\cite[Lemma~3.1]{DSS} stating the following, for $\xi_0\notin\mathcal D$ and $\rho>0$ such that $B_{4\rho}(\xi_0)\cap\mathcal D=\emptyset$.
If the image of $\nabla u$ stays outside $B_\rho(\xi_0)$, 
then, at a smaller scale,
 it must localize either completely inside $B_{4\rho}(\xi_0)$ or completely outside $B_{3\rho}(\xi_0)$. 
The first case already provides control on the oscillations of $\nabla u$, 
and in the second case one can iterate this lemma: 
if the second case keeps occuring, one eventually concludes that $\nabla u$ localizes inside a small neighborhood of $\mathcal D$.  When $\mathcal D$ is finite, this forces small oscillations.

The core idea in the proof of \cite[Lemma~3.1]{DSS} is Lebesgue's observation
that even though $H^1$ functions in the plane fail to be continuous,
continuity can be recovered if they satisfy a maximum and minimum principle 
(see e.g. \cite[\S~7.3-5]{IMGeom}).
This idea is applied to functions of $|\nabla u-\xi_0|$.
% for $\xi_0\notin\mathcal D$.
This requires an $H^1$ bound, and a maximum/minimum principle for such functions.
While the latter is somewhat general, the $H^1$ bound follows from a Cacciopoli-type inequality which heavily uses the fact that $\xi_0\notin\mathcal D$, and this is why \cite[Theorem~1.2]{DSS} is constrained to $\mathcal D$.

If, instead, we assume only that $\xi_0\notin \mathcal S$, there is no obvious reason why the $H^1$ bound should be valid.
A simple, but key, observation is that
there is a 
dual vector field $G^*$ such that 
$\eta_0=iG(\xi_0)\notin \mathcal D(G^*)$, and with the property that
 $iG(\nabla u)=\nabla v$
for a solution $v$ of
 $\dv G^*(\nabla v)=0$.
 This duality is already presented and exploited in \cite[\S~16.4]{AIM}.
Since it exchanges the roles of $\mathcal D$ and $\mathcal S$, one can apply the argument of
 \cite[Lemma~1.3]{DSS} to localize the image of $\nabla v$, and
come back to $\nabla u$ thanks to the strict monotony of $G$.
With this new localization lemma for $\xi_0\notin\mathcal S$, 
one can then follow the strategy of \cite[Theorem~1.2]{DSS}, 
as described above, with $\mathcal D$ replaced by $\mathcal D\cap\mathcal S$.

All
 these arguments are performed at the level of \emph{a priori} estimates, that is, assuming that $u$ is smooth, and we combine them with an approximation argument in order to conclude.

\subsection{Plan of the paper}

In Section~\ref{s:apriori}, we prove the a priori estimates described above. 
In Section~\ref{s:proof}, we perform the approximation argument to prove Theorem \ref{t:gen}. In Section~\ref{s:beltrami}, we transpose it to Beltrami equations, proving Theorem \ref{t:beltrami}. 
In Section~\ref{s:tildeS}, we prove Proposition~\ref{p:StildeS}.
In Section~\ref{s:ex} we construct the example of Theorem~\ref{t:ex}.
Several technical results are gathered in the Appendices.

\subsubsection*{Acknowledgements}

TL is supported by a EUR MINT PhD fellowship.
XL is supported by the ANR project ANR-22-CE40-0006.
This work was initiated during a stay
 at the Hausdorff Institute for Mathematics (HIM) in Bonn, funded by the Deutsche Forschungsgemeinschaft (DFG, German Research Foundation) under Germany’s Excellence Strategy –
EXC-2047/1 – 390685813, during the Trimester Program “Mathematics for Complex
Materials”.

\subsection{Notations}

%We will use the following notations :

\begin{align*}
\langle\cdot,\cdot\rangle \, \& \, \vert \cdot \vert
\quad & 
\text{ the scalar product \& associated euclidean norm on } \R^2\\
 B_r(x) \quad & \text{ the ball centered at $x$ with radius } r \\
 B_r \quad & \text{ the ball centered at 0 with radius }r\\ 
 i \quad  & \text{ the counterclockwise rotation of angle } \pi/2 \\
\interior A \quad & \text{ the interior of } A\subset\R^2 \\
\cl{A} \quad & \text{ the closure of } A\subset\R^2 \\
\partial A \quad & \text{ the boundary of } A\subset\R^2 \\
\diam A \quad & \text{ the diameter of } A \subset\R^2\\
\mathbb{S}^1 \quad & \text{ the unit circle } \lbrace x\in \R^2\colon \vert x \vert = 1 \rbrace \\
a \otimes b \quad & \text{ the  matrix with entries } (a\otimes b)_{ij}=a_ib_j\\
 D^{\zeta} \quad  & \text{ the finite difference operator } D^{\zeta} G(\xi) = G(\xi +\zeta)-G(\xi) \\
\nabla^s G \quad & \text{ the symmetric part }(\nabla G + \nabla G^T)/2  \text{ of }\nabla G
%\\
%\omega_G \quad
%&\text{ the modulus of monotony of }G
\end{align*}

%
%Finally, we recall Hadamard-Levy's Theorem :
%
%\begin{theorem}[Hadamard-Levy]
%Let $G : \R^n \to \R^n$, with $G \in \mathcal{C}^1(\R^n)$. Then the following are equivalents : 
%\begin{align*}
%& G \text{ is a } \mathcal{C}^1 \text{diffeomorphism of } \R^n \text{ onto } \R^n \\
%& G \text{ is coercive and } \nabla G(x) \text{ is invertible for all } x \in \R^n 
%\end{align*}
%\end{theorem}
%
%A direct consequence is that if $G$ is a smooth uniformly monotone vector field on $\R^2$, then $G$ is invertible, with smooth uniformly monotone inverse.

\section{A priori estimate}\label{s:apriori}

In this section we prove an explicit a priori estimate on the oscillation of $\nabla u$, in terms of:
\begin{itemize}
\item
the \emph{modulus of monotony} $\omega_G\colon (0,\infty)\to (0,\infty)$, given by
\begin{align*}
\omega_G(t)=\inf_{|\xi-\zeta| > t} \langle G(\xi)-G(\zeta),\xi-\zeta\rangle,
\end{align*}
\item and the open sets $O_\lambda(G)$, $V_\Lambda(G)$ given, as in \cite{DSS},
 by
\begin{align}
O_\lambda(G)
&
=
\interior
\left\lbrace
\xi\in\R^2\colon
\liminf_{\vert \zeta \vert \to 0} \frac{\langle G(\xi +\zeta)-G(\xi),\zeta \rangle}{\vert \zeta \vert^2} \geq \lambda
\right\rbrace
\, ,
\label{eq:Olambda}
\\
V_\Lambda(G)
&
=
\interior
\left\lbrace
\xi\in\R^2\colon
\liminf_{\vert \zeta \vert \to 0} \frac{\langle G(\xi+\zeta)-G(\xi),\zeta \rangle}{\vert G(\xi+\zeta)-G(\xi) \vert^2} \geq \frac{1}{\Lambda}
\right\rbrace
\, ,
\label{eq:VLambda}
\end{align}
for $\lambda,\Lambda>0$, where $\interior A$ denotes the topological interior of $A\subset\R^2$.
\end{itemize} 
The relevance of these sets with respect to Theorem~\ref{t:gen} is that we have
\begin{align*}
\mathcal D(G) =\R^2\setminus \bigcup_{\lambda >0} O_\lambda(G),
\qquad
\mathcal S(G)=\R^2\setminus \bigcup_{\Lambda>0} V_\Lambda(G).
\end{align*}
If
 $G$ is $C^1$, they coincide with
\begin{align*}
& O_{\lambda}(G) =  \interior 
\left \lbrace \xi \in \R^2
\colon  \langle \nabla^s G(\xi) \zeta,\zeta \rangle \geq \lambda \vert \zeta \vert^2,\;\forall \zeta\in\R^2 \right \rbrace\, ,  \\
& V_{\Lambda}(G) =  \interior 
\left \lbrace \xi \in \R^2
\colon \Lambda \langle \nabla ^s G(\xi) \zeta,\zeta\rangle \geq \left \vert \nabla G(\xi) \zeta \right \vert^2, \;\forall\zeta\in\R^2 \right \rbrace 
\, .
\end{align*}
The main result of this section is the following quantitative version of Theorem~\ref{t:gen}
for a priori smooth solutions, 
and a \emph{strongly monotone} vector field $G$, that is, there exists $C>0$ such that
\begin{align}\label{eq:strongmonot}
C \langle G(\xi)-G(\zeta),\xi-\zeta\rangle \geq |\xi-\zeta|^2 + |G(\xi)-G(\zeta)|^2,
\end{align}
for all $\xi,\zeta\in\R^2$. 
(This is equivalent to $G$ being both uniformly elliptic and globally Lipschitz.)

\begin{proposition}\label{p:apriori}
Let $G \colon \mathbb{R}^2 \to \mathbb{R}^2 $ smooth and strongly monotone.
Let $r>0$ and assume that there exist $\lambda,\Lambda,M>0$ and $\xi_1,\ldots,\xi_N\in\R^2$ with $|\xi_i-\xi_j|\geq 4r$ for $i\neq j$, and such that
\begin{align*}
\overline B_{2M} \subset  V_{\Lambda}(G) \cup O_{\lambda}(G) \cup \bigcup_{j=1}^N B_{r}(\xi_j)\, .
\end{align*} 
Then
any smooth solution $u$ of $\dv(G(\nabla u))=0$ in $B_1$
with $|\nabla u|\leq M$ satisfies
\begin{align*}
\diam(\nabla u(B_{\delta}))\leq r,
\end{align*}
 where
$\delta>0$ depends on
\begin{itemize}
\item a Lebesgue number $\eta\in (0,r)$ of the above open covering: 
any ball $B_\eta(\xi)$ with $|\xi|\leq 2M$ must be contained in $V_{\Lambda}(G)$, $O_{\lambda}(G)$ or $B_{r}(\xi_j)$ for some $j\in\lbrace 1,\ldots N\rbrace$;
\item  the gradient bound $M$ and 
the ellipticity constants $\lambda,\Lambda$;
\item the integrals $\int_{B_1}|\nabla u|^2\, dx $ and $\int_{B_1} |G(\nabla u)|^2\, dx$;
\item the modulus of monotony $\omega_G$ via any $c>0$ such that
$\omega_G(t)/t\geq c$ for all
 $t \in [\eta/4,M+\eta]$.
\end{itemize}
\end{proposition}

As explained in the introduction, the proof of this a priori estimate follows the strategy in \cite[Theorem~1.2]{DSS} and relies on two localization lemmas: 
\begin{itemize}
\item
one dealing with values of $\nabla u$ away from $\mathcal D$, that is, in $O_\lambda$, already proved in \cite[Lemma~3.1]{DSS};
\item 
a counterpart dealing with values of $\nabla u$ away from $\mathcal S$, that is, in $V_\Lambda$, which is our main new contribution.
\end{itemize}
We start by proving this new localization lemma 
before proceeding to the proof of Proposition~\ref{p:apriori}.

\subsection{The localization lemma in $V_\Lambda$}

This subsection is devoted to the proof of the following, 
where $G\colon\R^2\to\R^2$ 
is still assumed smooth and strongly monotone.
%: there exists $c>0$ such that
%\begin{align*}
%c|\xi-\zeta|^2\leq \langle G(\xi)-G(\zeta),\xi-\zeta\rangle \leq \frac 1c |\xi-\zeta|^2\quad\forall \xi,\zeta\in\R^2.
%\end{align*}

\begin{lemma}\label{l:locVLambda}
Let $u$ a solution of  $\dv(G(\nabla u))=0 $ in $B_1$ and assume that  
\[ \nabla u (B_1) \cap B_{\rho}(\xi_0) = \emptyset \text{ } \text{ and } \text{ } B_{4 \rho}(\xi_0)  \subset V_{\Lambda}(G)\, , \]
for some $\Lambda,\rho > 0$ and $\xi_0\in\R^2$.
Then we have
\begin{align*}
\text{either }
&
\nabla u(B_{\delta}) \subset B_{4 \rho}(\xi_0), 
%\\
\quad
\text{or }
%&
\nabla u (B_{\delta}) \cap B_{3 \rho}(\xi_0) = \emptyset \, ,
\end{align*}
for some $\delta>0$ depending on $\Lambda$, $\rho$, $\| \nabla u \|_{L^2(B_1)}$, 
and any $c>0$ such that $\omega_G(t)/t\geq c$ for $t\in [\rho, \|\nabla u\|_\infty+3\rho ] $.
\end{lemma}

As described in the introduction, the proof of Lemma~\ref{l:locVLambda} relies on 
a well-known duality \cite[\S~16.4]{AIM}. 
We recall here the basic properties that we will use.
Note that $G$ is invertible, as follows e.g. from the Minty-Browder theorem \cite[Theorem~9.14-1]{Ciarlet},
its inverse $G^{-1}$ is smooth thanks to the inverse function theorem,
and  the strong monotony \eqref{eq:strongmonot} of $G$ 
implies 
  that $G^{-1}$
 is strongly monotone.
At the core of the proof of Lemma~\ref{l:locVLambda} are the two following elementary but crucial properties.

\begin{lemma}\label{l:G*}
The dual vector field  $G^*(\xi)= i G^{-1}(-i\xi)$ is strongly monotone and
satisfies:
\begin{enumerate}
\item For any smooth solution $u$ of $\dv G(\nabla u)=0$ in $B_1$, there exists a smooth function $v$ such that $G(\nabla u)=-i\nabla v$ and
\begin{align*}
\dv G^*(\nabla v)=0\quad\text{in }B_1\, .
\end{align*}
\item For any $\Lambda>0$ and $\xi\in\R^2$ we have
\begin{align*}
\xi\in V_{\Lambda}(G)\quad\Longleftrightarrow
\quad
iG(\xi)\in O_{1/\Lambda}(G^*).
\end{align*}
\end{enumerate}
\end{lemma}
\begin{proof}
The strong monotony of $G^{-1}$ implies that of $G^*$. 
For the first property, the existence of $v$ such that $G(\nabla u)=-i\nabla v$ follows from Poincaré's lemma and the fact that  $iG(\nabla u)$ is curl-free. Then we see that
\begin{align*}
G^*(\nabla v)=G^*(iG(\nabla u))=iG^{-1}(G(\nabla u))=i\nabla u,
\end{align*}
is divergence-free.
The second property follows from the identity
\begin{align*}
\frac{\langle G(\xi +\zeta)-G(\xi),\zeta\rangle}{|G(\xi+\zeta)-G(\xi)|^2}
=
\frac{\langle G^*(\eta+\sigma)-G^*(\eta),\sigma\rangle}{|\sigma|^2}
\, ,
\end{align*}
 for any $\zeta\in\R^2\setminus\lbrace 0\rbrace$ and $\eta=iG(\xi)$, $\sigma=iG(\xi+\zeta)-iG(\xi)$, and the fact that $\zeta =G^{-1}(G(\xi)-i\sigma)-\xi \to 0$  if and only if $\sigma\to 0$.
\end{proof}

The duality provided by Lemma~\ref{l:G*} enables us to prove an $H^1$ estimate for $G(\nabla u)$ in the preimage  $(\nabla u)^{-1}(V_\Lambda)$, 
by reducing it to the following estimate in the preimage of $O_\lambda$.

\begin{lemma}\label{l:H1bdOlambda}
Let $u$ a smooth solution of $\dv(G(\nabla u))=0$ in $B_1$. Then we have
\begin{equation}
\label{Cacciopoli2bis}
\int_{B_{1/2} \cap (\nabla u)^{-1}(O_{\lambda})} \left \vert \nabla^2 u \right  \vert ^2 dx \leq C
\, , 
\end{equation}
where  $C = \frac{c_0}{\lambda^2} \| G( \nabla u) \|_{L^2(B_{1})}^2 $ for a universal constant $c_0>0$.
\end{lemma} 
\begin{proof}[Proof of Lemma~\ref{l:H1bdOlambda}]
This is proved in \cite[Proposition~3.3]{DSS} in the variational case $G=\nabla F$.
We provide here the minor changes needed to deal with a general strongly monotone $G$.

Multiplying the equation $\dv(\nabla G(\nabla u)\nabla u_k)=0$ (where the subscript $k$ denotes differentiation with respect to $x_k$) by $\xi^2 G^k(\nabla u)$ with $\xi$ a smooth cut-off function satisfying $\mathbf 1_{B_{1/2}}\leq\xi\leq\mathbf 1_{B_1}$,
and performing the same manipulations as in \cite[Proposition~3.3]{DSS}, 
using in particular that the matrix $\nabla G(\nabla u)\nabla^2u$ is trace-free, we obtain
\begin{align*}
2 \int_{B_1}\xi^2 \det(\nabla G(\nabla u)) \lvert \det(\nabla^2u) \rvert 
\,dx 
&
=  2\int_{B_1}G^i(\nabla u ) G^k(\nabla u) \partial_k(\xi_i \xi)
\, dx 
\\
&\leq c_0 \int_{B_1} |G(\nabla u)|^2\, dx.
\end{align*}
For any $A\in \R^{2\times 2}$ and $A^s$ its symmetric part, we have
\[  \det(A^{s}) = a_{11}a_{22}-\left(\frac{a_{12}+a_{21}}{2}\right)^2 \leq a_{11}a_{22}-a_{12}a_{21} = \det(A),\]
so we infer
\begin{align}\label{eq:H1bd1}
2 \int_{B_1}\xi^2 \det(\nabla^s G(\nabla u)) \lvert \det(\nabla^2u) \rvert 
\,dx 
&\leq c_0 \int_{B_1} |G(\nabla u)|^2\, dx.
\end{align}
Then, using that the matrix 
\begin{align*}
A= \left(\nabla^{s}G(\nabla u) \right)^{\frac{1}{2}} \nabla^2u \left(\nabla^{s}G(\nabla u) \right)^{\frac{1}{2}},
\end{align*}
is symmetric and trace-free we obtain, 
arguing as in \cite[Proposition~3.3]{DSS}, the inequality
\begin{align*}
\det (\nabla^{s}G(\nabla u) )\left \vert \det \nabla^2u \right \vert \geq \frac{\lambda^2}{2} \left \vert \nabla^2u \right \vert ^2 \mathbf 1_{ (\nabla u)^{-1}(O_{\lambda}(G))},
\end{align*} 
which, plugged back into \eqref{eq:H1bd1}, concludes the proof.
\end{proof}

We
 combine the estimate of Lemma~\ref{l:H1bdOlambda} and the duality of Lemma~\ref{l:G*} to obtain an estimate in the preimage of $(\nabla u)^{-1}(V_\Lambda)$.

\begin{lemma}\label{l:H1bdVLambda}
Let $u$ a smooth solution of $\dv(G(\nabla u))=0$ in $B_1$. Then we have 
\begin{equation*}
\int_{B_{1/2} \cap (\nabla u)^{-1}(V_{\Lambda}(G))} \left \vert \nabla \left (G( \nabla u) \right) \right  \vert ^2 dx \leq C 
\end{equation*}
Where $C = c_0 \Lambda^2 \| \nabla u \|_{L^2(B_{1})}^2$  for a universal constant $c_0>0$.
\end{lemma} 

\begin{proof}
By Lemma~\ref{l:G*} we have $iG(\nabla u)=\nabla v$,
where $v$ satisfies
 \begin{align*}
 \dv(G^*(\nabla v))=0\quad\text{ in }B_1,
 \quad
 (\nabla v)^{-1} \left (O_{\frac{1}{\Lambda}}(G^*) \right)=
 (\nabla u )^{-1} \left( V_\Lambda(G) \right ) \, ,
 \end{align*} 
 hence Lemma~\ref{l:H1bdVLambda} follows from 
 Lemma~\ref{l:H1bdOlambda} applied to $v$ and $G^*$.
\end{proof}

We are ready to prove Lemma~\ref{l:locVLambda}:

\begin{proof}[Proof of Lemma~\ref{l:locVLambda}]
%Let $\delta> 0$ and assume that there exists $x_0,x_1 \in B_{\delta}$ such that : 
%\[ \left \vert \nabla u(x_0)-p_0 \right \vert > 4 \rho \quad \text{and } \left \vert \nabla u(x_1)-p_0 \right \vert < 3 \rho.\]
As remarked in \cite[Proposition~3.6]{Lledos}, we have the 
maximum/minimum principle
\begin{align}\label{eq:maxminpple}
\partial (\nabla u(B_r) ) \subset \nabla u (\partial B_r),
\end{align}
which follows from \cite[Theorem~II]{HN59} since 
$\det(\nabla^2 u)$ does not change sign (as a consequence of 
$\nabla^s G(\nabla u)\nabla^2 u$ being trace-free).

We denote $M=\|\nabla u\|_\infty$ and prove Lemma~\ref{l:locVLambda} by contradiction:
assume that  $\delta\in (0,1/2] $
is such that $\nabla u(B_\delta)$ intersects both $B_{3\rho}(\xi_0)$ and
$\overline B_M \setminus B_{4\rho}(\xi_0)$.

Since 
$\nabla u(B_1) \cap B_{\rho}(\xi_0) = \emptyset $,
this implies that, for any $r\in (\delta,1)$ the boundary of
$\nabla u(B_r)$
 intersects both 
$B_{3\rho}(\xi_0)$ and
$\overline B_M\setminus B_{4\rho}(\xi_0)$.  
Indeed, $\nabla u(B_r)$ contains a point $\zeta$ in $B_{3\rho}(\xi_0)$ and does not contain $\xi_0$, so on the segment $[\xi_0,\zeta]\subset B_{3\rho}(\xi_0)$ there must be a point belonging to the boundary of $\nabla u(B_r)$.
Similarly,
$\nabla u(B_r)$ contains a point $\zeta$ in $\overline B_M\setminus B_{4\rho}(\xi_0)$,
and, since $\nabla u(B_r)\subset\overline B_M$,
 on the half line $\lbrace \zeta +t(\zeta-\xi_0)\rbrace_{t\geq 0}\subset \R^2\setminus B_{4\rho}(\xi_0)$ there must be a point
 belonging to the boundary of $\nabla u(B_r)$ (and then automatically also to $\overline B_M$).

Thanks to the maximum/minimum principle \eqref{eq:maxminpple}
we deduce that
$\nabla u(\partial B_r)$
 intersects both 
$B_{3\rho}(\xi_0)$ and
$\overline B_M\setminus B_{4\rho}(\xi_0)$.
Define the sets
\begin{align*}
 \overline{\Sigma}= G \left( \overline B_M \setminus B_{4\rho}(\xi_0) \right) \quad 
 \underline{\Sigma} =  G \left( B_{3\rho}(\xi_0) \right)   
 \, ,
\end{align*}
and note that, by definition of the modulus of monotony $\omega_G$, we have
\begin{align*}
\dist(\overline\Sigma,\underline\Sigma)
\geq \eta := \inf_{\rho\leq t\leq M+ 3\rho}\frac{\omega_G(t)}{t} >0. 
\end{align*}

Let $\mathcal{R} \in C^1(\R^2)$  such that 
 $|\nabla \mathcal R|\leq  \frac{4}{\eta}\mathbf 1_{\R^2\setminus(\overline\Sigma \cup \underline{\Sigma})}$ and
\begin{align*}
 \mathcal{R}(\xi) = 
\begin{cases}
0 &\text{ if }\xi\in\overline\Sigma,\\
1 &\text{ if }\xi\in\underline\Sigma.
\end{cases}
\end{align*}
Recall that
$\nabla u(\partial B_r)$
 intersects both 
$B_{3\rho}(\xi_0)$ and
$\overline B_M\setminus B_{4\rho}(\xi_0)$.
This implies that $\mathcal R(G(\nabla u)))$ takes both the value 0 and the value 1 on $\partial B_r$, and therefore, by the mean value theorem,
\begin{align*}
1\leq \int_{\partial B_r} \left \vert \nabla \left( \mathcal{R}(G(\nabla u)) \right) \right \vert \, ds
\leq 
\sqrt{2 \pi r} \left (\int_{\partial B_r} \left \vert \nabla \left( \mathcal{R}(G(\nabla u)) \right) \right \vert^2\, ds \right)^{\frac{1}{2}}
\end{align*}
Dividing by $\sqrt r$, squaring and and integrating it on $[ \delta, 1/2 ]$ we find
\begin{align*}
\log\left(\frac{1}{2\delta} \right)
&
\leq 2\pi 
\int_{B_{1/2}} \left \vert \nabla \left( \mathcal{R}(G(\nabla u)) \right) \right \vert^2 dx
\\
&
\leq 
\frac{32\pi}{\eta^2}
\int_{B_{1/2} \cap (\nabla u)^{-1}(B_{4 \rho}(\xi_0))}  \left \vert \nabla (G(\nabla u)) \right \vert^2 dx.
\end{align*}
The last inequality follows from the chain rule, the fact that $|\nabla\mathcal R|=0$  on $\overline\Sigma=G( \overline B_M \setminus B_{4\rho}(\xi_0) )$, and the inequality $|\nabla\mathcal R|\leq 4/\eta$.
Recalling that $B_{4\rho}(\xi_0)\subset V_\Lambda(G)$, we can use Lemma~\ref{l:H1bdVLambda} 
to deduce
\begin{align*}
\log\left(\frac{1}{2\delta} \right)
\leq \frac{32\pi}{\eta^2} 
\int_{B_{1/2} \cap (\nabla u)^{-1}(V_\Lambda(G))} 
\left \vert \nabla (G(\nabla u)) \right \vert^2 dx
\leq\frac{32\pi C}{\eta^2}.
\end{align*}
For $\delta < \exp(-32\pi C/\eta^2)/2$ this is impossible, 
and the conclusion of Lemma~\ref{l:locVLambda} is therefore verified.
\end{proof}

\subsection{Proof of the a priori estimate}

Before proving  Proposition~\ref{p:apriori}, we recall the localization lemma \cite[Lemma~3.1]{DSS} near $O_\lambda$.

\begin{lemma}\label{l:locOlambda}
Let $u$ a solution of  $\dv G(\nabla u)=0 $ in $B_1$ and assume that  
\[ \nabla u (B_1) \cap B_{\rho}(\xi_0) = \emptyset \text{ } \text{ and } \text{ } B_{4 \rho}(\xi_0)  \subset O_{\lambda}(G)\, , \]
for some $\lambda,\rho > 0$ and $\xi_0\in\R^2$.
Then we have
\begin{align*}
\text{either }
&
\nabla u(B_{\delta}) \subset B_{4 \rho}(\xi_0), 
%\\
\quad
\text{or }
%&
\nabla u (B_{\delta}) \cap B_{3 \rho}(\xi_0) = \emptyset \, ,
\end{align*}
for some $\delta>0$ depending on $\lambda$, $\rho$, and $\| G(\nabla u) \|_{L^2(B_1)}$.
\end{lemma}

In \cite{DSS} this is proved in the variational setting $G=\nabla F$, 
but the only step that needs minor adaptation is the $H^1$ estimate \cite[Proposition~3.3]{DSS},
which we have adapted here in Lemma~\ref{l:H1bdOlambda}.
Then the proof of Lemma~\ref{l:locOlambda} 
is completed using arguments similar to Lemma~\ref{l:locVLambda}, 
based on the maximum/minimum principle \eqref{eq:maxminpple} and the estimate of oscillations on the circles $\partial B_r$.

\begin{proof}[Proof of Proposition \ref{p:apriori}]

Let $u$ be a smooth  solution of $\dv G(\nabla u)=0 $
 in $B_1$ with $|\nabla u|\leq M$.
By assumption we have
\begin{align*} 
\overline B_{2M}\subset  V_{\Lambda}(G) \cup O_{\lambda}(G)\cup \bigcup_{j=1}^N B_{r}(\xi_j)
\, ,
\end{align*} 
and we fix a Lebesgue number $\eta\in (0,r)$ of this open covering, with the property that
any ball $B_\eta(\xi)$ centered at $\xi\in \overline B_{2M}$ is contained in 
$V_{\Lambda}(G)$, $O_{\lambda}(G)$, or $ B_{r}(\xi_j)$ for some $j\in\lbrace 1,\ldots,N\rbrace$. 
We set $\rho=\eta/4$.
The ball $\overline B_{2M}$ can be covered by a finite number of balls of radius $\rho$.
Removing the balls that are contained in one of the $B_r(\xi_j)$, we are left with a covering
\begin{align*}
\overline B_{2M}\setminus \bigcup_{j=1}^N B_{r}(\xi_j) \subset \bigcup_{k=1}^K B^k_{\rho},
\end{align*}
with $K\leq c M^2/\eta^2$ for some universal constant $c>0$, and the property that each ball $B^k_{4\rho}$ satisfies
\begin{align*}
B^k_{4\rho}\subset V_\Lambda(G)
\quad
\text{or}
\quad 
B^k_{4\rho}\subset O_\lambda(G).
\end{align*}
Since $ \nabla u (B_1) \subset \overline{B_M} $ , there exists a ball $B^k_{4\rho} \subset B_{2M} \setminus \overline{B_M} $ such that : 
\[ \nabla u (B_1) \cap B_{\rho}^k = \emptyset           \]
Since
 $ B^k_{4\rho} \subset V_{\Lambda}(G)$ or $B^k_{4\rho} \subset O_{\lambda}(G) $, 
 we can apply Lemma~\ref{l:locVLambda} or Lemma~\ref{l:locOlambda} 
  to ensure the existence of some $\delta>0$ such that
\begin{align*}
\text{either }\nabla u(B_\delta)\subset B_{4\rho}^k
\quad\text{or }
\nabla u(B_\delta)\cap B_{3\rho}^k=\emptyset.
\end{align*}  
If the first case occurs, then we are done since $4\rho=\eta <r$. 
If the second case occurs, 
we infer that 
$ \nabla u(B_{\delta}) \cap B^j_{\rho} = \emptyset $ for all neighboring balls $B_{\rho}^j$ such that $ B_{\rho}^j \cap B_{\rho}^k \neq \emptyset $. 
Then we can apply again Lemma~\ref{l:locVLambda} or Lemma~\ref{l:locVLambda}
to the rescaled function $\delta^{-1}u(\delta \cdot)$ 
and these neighboring balls $B_{\rho}^j$. 

We iterate this argument: if at some step we reach the first case, we are done.
Otherwise, since $\overline B_{2M}\setminus \bigcup_{j=1}^N B_r(\xi_j)$ is connected, 
we eventually cover it with the neighboring balls added at each step, and deduce that
$\nabla u(B_{\delta'})\subset \bigcup_{j=1}^N B_r(\xi_j)$.
Here $\delta'=\delta^K$ for $\delta$ as in Lemma~\ref{l:locVLambda} and Lemma~\ref{l:locOlambda}.
By connectedness, $\nabla u(B_{\delta'})$ is contained in one of the balls $B_r(\xi_j)$, 
and this concludes the proof.
 \end{proof}

\begin{remark}\label{r:connect}
If, in the assumptions of Proposition~\ref{p:apriori}, the union of the balls $B_r(\xi_j)$ is replaced by any open subset $U\subset\R^2$ such that $\overline B_{2M}\setminus U$ is connected, then the same proof shows, for any $r>0$, the existence of $\delta >0$ such that either $\diam(\nabla u(B_\delta))<r$ or $\nabla u(B_\delta)\subset U$.
\end{remark}

\section{Proof of Theorem~\ref{t:gen}}\label{s:proof}

In this section, we prove Theorem~\ref{t:gen} using the a priori estimate of Proposition~\ref{p:apriori} and an approximation argument. 
This is quite standard, see e.g. \cite{CF14a,Lledos}, 
but some details seem needed to make sure it applies in our situation.
For the reader's convenience we recall here the statement of Theorem~\ref{t:gen} .

\begin{theorem}
Let $G \colon \mathbb{R}^2 \to \mathbb{R}^2$ a continuous strictly monotone vector field such that  $\mathcal{S}(G) \cap \mathcal{D}(G) \cap \overline{B_M}$ is a finite set. 
Then any $M$-Lipschitz solution $u$ of $\dv G(\nabla u)=0 $ in $B_1$ is $C^1$.
\end{theorem}

\begin{proof}
Let $u$ and $G$ as in the Theorem. 
Modifying $G$ ouside $\overline B_M$ does not change the equation satisfied by $u$.
Thanks to Lemma~\ref{l:modif} 
we can therefore assume that 
\begin{align*}
\mathcal D(G)\cap \mathcal S(G)\subset \overline B_M,
\end{align*}
and apply Lemma~\ref{l:approxG},
to obtain smooth strongly monotone vector fields $G_\e$ such that
$\omega_{G_{\e}} \geq \omega_G$ and
\begin{align}\label{eq:GepsOV}
\begin{aligned}
&
B_{2\e}(\xi)\subset O_\lambda(G)\;\Rightarrow\; \xi\in O_\lambda(G_\e)\, ,
\\
& 
B_{2\e}(\xi)\subset V_{\Lambda}(G)\;\Rightarrow\; \xi\in V_{\Lambda+\e}(G_\e)\, .
\end{aligned}
\end{align}
Thanks
 to Lemma~\ref{l:approxu}, the smooth solutions $u_\e$ of 
\begin{align*}
\dv G_\e(\nabla u_\e)=0\;\text{ in }B_1,\quad u_\e=u\;\text{ on }\partial B_1,
\end{align*}
satisfy $|\nabla u_\e|\leq \widetilde M$ in $B_{1/2}$ for some $\widetilde M\geq M$, and converge to $u$ in $H^1(B_1)$.

Rescaling $B_{1/2}$ to $B_1$, we assume that $|\nabla u_\e|\leq \widetilde M$ in $B_1$ and apply Proposition~\ref{p:apriori} to $u_\e$.
For any given $r>0$, we will check that the radius $\delta>0$ obtained that way,
such that $\diam(\nabla u_\e(B_\delta))<r$, does not depend on $\e$.
Passing to the limit will then prove continuity of $\nabla u$ at $0$ (and, 
translating and rescaling,
 at any point $x\in B_1$).

Denote
\begin{align*}
\mathcal D(G)\cap\mathcal S(G) =\lbrace \xi_1,\ldots,\xi_N\rbrace
\subset\overline B_{\widetilde M}.
\end{align*}
Let $r>0$. 
Since $\mathcal D$ is the complement of $\bigcup_{\lambda>0} O_\lambda$ and $\mathcal S$ the complement of $\bigcup_{\Lambda>0} V_\Lambda$, we can find $\lambda,\Lambda>0$ such that
\begin{align*}
\overline B_{2\widetilde M}
\subset O_\lambda(G)\cup V_\Lambda(G) \cup \bigcup_{j=1}^N B_{r}(\xi_j).
\end{align*}
We may also fix a Lebesgue number $\eta\in (0,r/2)$ 
such that
any ball 
$B_{4\eta}(\xi)$ 
with $|\xi|\leq 2\widetilde M$
must be contained in 
$O_\lambda(G)$ or 
$V_\Lambda(G)$
or one of the balls
$B_{r}(\xi_j)$ for some $j=1,\ldots,N$.
Thanks to the properties \eqref{eq:GepsOV} of $G_\e$, 
for any $0<\e<\min(\eta,\Lambda)$  we have
\begin{align*}
\overline B_{2\widetilde M}
\subset O_\lambda(G_\e)\cup V_{2\Lambda}(G_\e) \cup \bigcup_{j=1}^N B_{r}(\xi_j),
\end{align*}
and $\eta$ has the Lebesgue number property that
any ball 
$B_{\eta}(\xi)$ 
with $|\xi|\leq 2\widetilde M$
must be contained in 
$O_\lambda(G_\e)$ or 
$V_\Lambda(G_\e)$
or one of the balls
$B_{r}(\xi_j)$ for some $j=1,\ldots,N$.
Moreover, 
if $c>0$ is such that $\omega_G(t)/t\geq c$ for all $t\in [\eta/4,\widetilde M+\eta]$, then
we have $\omega_{G_\e}(t)/t\geq \omega_G(t)/t \geq c$ for all  $t\in [\eta/4,\widetilde M+\eta]$.
Applying Proposition~\ref{p:apriori} we obtain therefore a radius $\delta>0$, independent of $\e$, such that
$\diam(\nabla u_\e(B_\delta))\leq r$.
\end{proof}

\section{Nonlinear Beltrami equations}\label{s:beltrami}

In this section we describe how  
 to transform Theorem~\ref{t:gen} about degenerate elliptic equations $\dv G(\nabla u)=0$ into Theorem~\ref{t:beltrami} about degenerate Beltrami equations $f_{\bar z}=H(f_z)$.
 This relies on Minty's correspondence \cite{Minty} and is described thoroughly in \cite[\S~16]{AIM}.
 For the readers' convenience, we recall here and sketch the proof of
 the basic features that we are going to use.

\begin{proposition}\label{p:H}
Let  $H\colon\C\to\C$ a strictly $1$-Lipschitz function, that is,
\begin{align*}
|H(\xi)-H(\zeta)|<|\xi -\zeta|\quad\forall \xi\neq\zeta\in\C.
\end{align*}
Then:
\begin{enumerate}
\item One may modify $H$ outside any arbitrary compact in order to ensure
\begin{align}\label{eq:Hinfty}
\lim_{|z|\to \infty} 
\left(
|z| \pm \langle \overline H(z),\frac{z}{|z|}\rangle 
\right)= +\infty\, ,
\end{align}
which we assume from now on.
\item The maps $F,F_*\colon \mathbb C\to\C$ given by
\begin{align}\label{eq:FF*}
F( z)= \frac{ H(z) +\bar z }{2},
\quad 
F_*(z)=\frac{ H(z)-\bar z}{2i},
\end{align}
are homeomorphisms.
\item The maps $G,G^*$ given by
\begin{align}\label{eq:GF}
G=-i F_*\circ F^{-1},\qquad G^* = i F \circ F_*^{-1},
\end{align}
are continuous, strictly monotone vector fields,
and for any complex function $f\colon B_1\to\C$, we have the implication
\begin{align}\label{eq:fuv}
f_{\bar z}=H(f_z)
\quad\Longrightarrow\quad
\dv G(\frac 12 \nabla u) =\dv G^*(\frac 12 \nabla v)=0\ ,
\end{align}
where $u=\re f$ and $v=\im f$.
\item Under this correspondence, the sets $\mathcal D,\mathcal S$ \eqref{eq:DS}
are transformed as
\begin{align}\label{eq:GammaDS}
\begin{aligned}
&
\mathcal D(G)= F(\Gamma_+),\quad \mathcal S(G)=F(\Gamma_-),
\\
&
\mathcal D(G^*)= F_*(\Gamma_-),\quad \mathcal S(G^*)=F_*(\Gamma_+).
\end{aligned}
\end{align}
where $\Gamma_\pm$ are as in Theorem~\ref{t:beltrami}.
\end{enumerate}
\end{proposition}

\begin{proof}[Proof of Proposition~\ref{p:H}]

1. 
For any $R>0$ one may pick a smooth function $\chi\colon [0,\infty)\to [0,1]$ such that $\chi\equiv 1$ on $[0,R]$, $-1\leq r\chi'(r)\leq 0$ for all $r\geq 0$,  and $\chi(r)\to 0$ as $r\to +\infty$.
Then the map $\Phi\colon \C\to\C$ given by $\Phi(z)=\chi(|z|)z$ 
equals the identity in $B_R$, 
is 1-Lipschitz because its differential at $z=re^{i\theta}$ is symmetric with eigenvalues $\chi(r)$ and $\chi(r)+ r\chi'(r)$,
and $\Phi(z)\to 0$ as $|z|\to +\infty$.
Thus $\widetilde H =H\circ\Phi$ equals $H$ in $B_R$, is strictly 1-Lipschitz, and $H(z)\to H(0)$ as $|z|\to +\infty$, which implies \eqref{eq:Hinfty}.

2. To check that $F,F_*$ are homeomorphisms,
one can remark that $\overline F$ and $i\overline F_*$  are strictly monotone and continuous,
and thanks to \eqref{eq:Hinfty} they are coercive:
\begin{align*}
\lim_{|z|\to +\infty} \frac{\langle \overline F(z),z\rangle}{|z|} 
=
\lim_{|z|\to +\infty} \frac{\langle i\overline F_*(z),z\rangle}{|z|}
=+\infty
\,
.
\end{align*}
Hence the Minty-Browder theorem 
\cite[Theorem~9.14-1]{Ciarlet} ensures that they are invertible.
Continuity of their inverses is also a consequence of the coercivity:
if a sequence $(z_k)$ is such that $F(z_k)\to\xi$,
then coercivity forces $(z_k)$ to be bounded, and
by continuity of $F$ any converging subsequence must converge to $F^{-1}(\xi)$.

3. 
Note that $G,G^*$ are dual to each other in the sense of Lemma~\ref{l:G*}, that is,
$-iG^*(iG(\xi))=iG(-iG^*(\xi))=\xi$ for all $\xi\in\C$.

Continuity of $G,G^*$ follows from the previous item.
For any $\xi\in \C$ and $\zeta\neq 0$, 
letting 
 $\eta=F(\xi)$ and
$\sigma=F(\xi+\zeta)-F(\xi)$, we have
\begin{align*}
\langle G(\eta +\sigma)-G(\eta),\sigma\rangle
=|\zeta|^2-|H(\xi+\zeta)-H(\xi)|^2 >0,
\end{align*}
so $G$ is strictly monotone, and similarly for $G^*$.
The implication
 \eqref{eq:fuv} 
follows by rewriting 
$f_{\bar z}=H(f_z)$, 
as 
\begin{align*}
2u_{\bar z} =H(f_z) +\overline{f_{z}}
\text{ and }
2i v_{\bar z} = H(f_z)-\overline{f_{z}},
\end{align*} 
that is,
\begin{align*}
u_{\bar z}= F(f_z)
\text{ and }
v_{\bar z}= F_*(f_z),
\end{align*}
or equivalently 
\begin{align*}
G(u_{\bar z})=-iv_{\bar z}
\text{ and }G^*(v_{\bar z})=iu_{\bar z},
\end{align*}
which are divergence free. 
More details can be found e.g. in \cite[Theorem~5]{AstalaClop}.

4.
For any $\xi\in \C$ and $\zeta\neq 0$, 
letting 
 $\eta=F(\xi)$ and
$\sigma=F(\xi+\zeta)-F(\xi)$,
we have the identities
\begin{align*}
\frac{\langle G(\eta +\sigma)-G(\eta),\sigma\rangle}{|\sigma|^2}
&
%=
%\frac{\langle -i F_*(\xi +\zeta)+ i F_*(\xi),F(\xi+\zeta)-F(\xi)\rangle}{|F(\xi+\zeta)-F(\xi)|^2}
%\\
%&
%=
%\frac{\langle\bar\zeta - H(\xi +\zeta)+ H(\xi),H(\xi+\zeta)-H(\xi) +\bar \zeta\rangle}
%{|H(\xi+\zeta)-H(\xi) +\bar \zeta|^2}
%\\
%&
=\frac{1-|L_H(\xi,\zeta)|^2}{|1+L_H(\xi,\zeta)|^2}\, ,
\\
\frac{\langle G(\eta +\sigma)-G(\eta),\sigma\rangle}{|G(\eta +\sigma)-G(\eta)|^2}
&
=\frac{1-|L_H(\xi,\zeta)|^2}{|1-L_H(\xi,\zeta)|^2}\, ,
\end{align*}
where
\begin{align*}
L_H(\xi,\zeta)=\frac{H(\xi+\zeta)-H(\xi)}{\bar\zeta}.
\end{align*}
Recall moreover that
$F$ is a homeomorphism and
$\sigma=F(\xi+\zeta)-F(\xi)\to 0$ if and only if 
$\zeta=F^{-1}(F(\xi)+\sigma)-\xi \to 0$.
Therefore, these identities and the definitions \eqref{eq:DS} of $\mathcal D,\mathcal S$ 
 imply that
\begin{align}
\mathcal D(G)= F(\Gamma_+),\quad \mathcal S(G)=F(\Gamma_-),
\end{align}
with
\begin{align*}
\Gamma_\pm =\bigcap_{\lambda>0} \cl \left\lbrace
\xi\in\C\colon \liminf_{|\zeta|\to 0} \frac{1-|L_H(\xi,\zeta)|^2}{|1\pm L_H(\xi,\zeta)|^2}  \leq \lambda \right\rbrace,
\end{align*}
as in Theorem~\ref{t:beltrami}.
Similar calculations (or the duality of Lemma~\ref{l:G*}) give
$ \mathcal D(G^*)= F_*(\Gamma_-)$ and $\mathcal S(G^*)=F_*(\Gamma_+)$.
\end{proof}

\begin{remark}\label{r:minty}
Reciprocally, if $G\colon\R^2\to\R^2$ is continuous and strictly monotone, 
the Minty-Browder theorem ensures that $\psi\colon\R^2\to\R^2$ given by
\begin{align*}
 \psi(\xi) =\frac{\xi+G(\xi)}{2},
\end{align*} 
is a homeomorphism, so is its pointwise conjugate $\phi=\overline\psi$,
 and then the map $H\colon \C\to\C$ given
\begin{align*}
H(\phi(\xi))=\frac{\xi-G(\xi)}{2},
\end{align*}
is strictly $1$-Lipschitz. 
If it satisfies \eqref{eq:Hinfty}, then $G$ can be recovered as in Proposition~\ref{p:H}. 
\end{remark}

Thanks to \eqref{eq:fuv} and \eqref{eq:GammaDS},
it becomes apparent that
 Theorem~\ref{t:beltrami} is a consequence of Theorem~\ref{t:gen}.
In fact, keeping track of the role of $M$, we obtain the following  more precise version of Theorem~\ref{t:beltrami}.

\begin{theorem}\label{t:beltramiM}
Let $H \colon \mathbb{C} \to \mathbb{C}$ strictly $1$-Lipschitz, and $M>0$ such that 
$F(\Gamma_+)\cap F(\Gamma_-)\cap \overline B_M(\xi_0/2)$ and 
$F_*(\Gamma_+)\cap F_*(\Gamma_-)\cap \overline B_M(\xi_0/2i)$ are finite, where $\xi_0=H(0)$.
 Then any 
Lipschitz
 solution $f$ of $ f_{\bar{z}} = H(f_z)$ in $B_1$ with $|f_z|\leq M$ is $C^1$.
\end{theorem}

\begin{proof}
Replacing $H$ by $H-\xi_0$ and $f$ by $f-\xi_0\bar z$ we assume without loss of generality that $H(0)=0$.
Then the property $|f_z|\leq M$ implies $|f_{\bar z}|\leq M$.
Writing $f=u+iv$ we deduce 
$|u_{\bar z}| \leq M$ and $|v_{\bar z}|\leq M$, that is,
$|\nabla u|\leq 2M$ and $|\nabla v|\leq 2M$.

Thanks to the implication \eqref{eq:fuv}, and since $F(\Gamma_+)\cap F(\Gamma_-)=\mathcal D(G)\cap\mathcal S(G)$,
we can apply Theorem~\ref{t:gen} to $u/2$ and the vector field $G$, so that $u$ is $C^1$.
Similarly we obtain that $v$ is $C^1$ and conclude that $f$ is $C^1$.
\end{proof}

\begin{remark}
It can be instructive to contemplate the correspondence
\eqref{eq:GammaDS}
 in the case of the $p$-Laplacian.
 For  $G(\xi) = \vert \xi \vert^{p-2} \xi$, 
 we have on the one hand $ \mathcal{D}(G) = \lbrace 0 \rbrace$, $\mathcal{S}(G)= \emptyset $ if $p>2$ and $\mathcal{D}(G) = \emptyset$, $\mathcal{S}(G)=\lbrace 0 \rbrace $ for $p<2$. 
On the other hand, 
with the bijection $\phi(\xi)=(\overline{\xi +G(\xi)})/2$ as in Remark~\ref{r:minty},
we have
 \[ \frac{H(\phi(z))-H(\phi(0))}{\overline{\phi(z)-\phi(0)}} = \frac{1-\vert z \vert^{p-2}}{1+\vert z \vert^{p-2}} \in \mathbb{R}                 \]
As $z\to 0$, this quantity goes to $\pm 1$ depending on the value of $p$. 
In this case, $H$ acts like $\pm$ the conjugation around the origin, and $L_H \in \R$. 
Depending on the value of $p$ we can check that $(\Gamma_+,\Gamma_-)=(\emptyset,\lbrace 0\rbrace)$ or $(\lbrace 0\rbrace,\emptyset)$, in accordance with \eqref{eq:GammaDS}.
\end{remark}

\section{Proof of Proposition~\ref{p:StildeS}}\label{s:tildeS}

In this section we consider $G\colon\R^2\to\R^2$ continuous strictly monotone, and
we prove the 
basic property  stated in
 Proposition~\ref{p:StildeS} :
 \begin{align}\label{eq:StildeS}
\widetilde{\mathcal S}(G)
\subset
{\mathcal S}(G)\, ,
 \end{align}
with equality if $G=\nabla F$.

The inclusion \eqref{eq:StildeS}
is a consequence of Cauchy-Schwarz' inequality : 
for all $\xi\in\R^2$ and $\zeta\neq 0$ we have
\[ \frac{\langle G(\xi+\zeta)-G(\xi),\zeta \rangle}{\vert G(\xi+\zeta)- G(\xi) \vert^2} \leq \frac{\vert \zeta \vert^2}{\langle G(\xi+\zeta)-G(\xi),\zeta \rangle}\,
,
 \] 
and the conclusion follows by taking the liminf as $\zeta \to 0 $
and recalling the definitions \eqref{eq:DS} and \eqref{eq:tildeS} of $\mathcal S$ and $\widetilde{\mathcal S}$.

Next, we assume that $G=\nabla F$, 
 fix $\xi_0\in \R^2\setminus \widetilde{\mathcal S}(\nabla F)$,
and prove that $\xi_0\notin \mathcal S(\nabla F)$,
which implies equality in \eqref{eq:StildeS}.

By definition \eqref{eq:tildeS} of $\widetilde{\mathcal S}$ , there exist $\Lambda,r>0$ such that
\begin{align*}
\limsup_{\zeta\to 0}
\frac{ \langle\nabla F(\xi+\zeta)-\nabla F(\xi),\zeta\rangle}{|\zeta|^2}\leq\Lambda\quad\forall \xi\in B_{3r}(\xi_0).
\end{align*}
Fix $\xi\in B_{2r}(\xi_0)$ and $\zeta\in B_{r}$, then the function $f\colon [0,1]\to\R$ given by
\begin{align*}
f(t)=\langle\nabla F(\xi +t\zeta),\zeta\rangle \quad\forall t\in [0,1],
\end{align*}
is monotone nondecreasing, and the above property of $F$ ensures that
\begin{align*}
\lim_{s\to 0^+}\frac{f(t+s)-f(t)}{s} \leq \Lambda |\zeta|^2,\quad\forall t\in [0,1].
\end{align*}
This implies that $f$ is absolutely continuous
and
$0\leq f'\leq \Lambda |\zeta|^2$. 
We infer that $f(1)-f(0)\leq \Lambda |\zeta|^2$, that is,
\begin{align*}
\frac{ \langle\nabla F(\xi+\zeta)-\nabla F(\xi),\zeta\rangle}{|\zeta|^2}\leq\Lambda,
\qquad\forall \xi\in B_{2r}(\xi_0),\;\forall\zeta\in B_{r}.
\end{align*}
Consider the mollified function 
$F_\e(\xi)=\int F(\xi +\e z)\rho(z)\, dz$, for some smooth nonnegative kernel
 $\rho\in C_c^\infty(B_1)$.
From the last inequality, we infer, for $0<\e<r/2$,
\begin{align*}
\frac{ \langle\nabla F_\e(\xi+\zeta)-\nabla F_\e(\xi),\zeta\rangle}{|\zeta|^2}\leq\Lambda,
\qquad\forall \xi\in B_{r}(\xi_0),\;\forall\zeta\in B_r.
\end{align*}
Letting $\zeta\to 0$, this implies
\begin{align*}
0\leq \nabla^2 F_\e(\xi) \leq \Lambda\qquad\forall \xi\in B_{r}(\xi_0).
\end{align*}
Since $\nabla^2 F_\e(\xi)$ is symmetric nonnegative, we infer
\begin{align*}
|\nabla^2 F_\e(\xi)|^2 \leq\Lambda \langle \nabla^2 F_\e(\xi)\zeta,\zeta\rangle \qquad\forall\xi\in B_r(\xi_0),\;\forall\zeta\in\R^2.
\end{align*}
Using also Jensen's inequality,
this implies,
 for all $\xi \in B_{r/2}(\xi_0)$ and $\zeta \in B_{ r/2}$, 
\begin{align*}
|\nabla F_\e(\xi+\zeta)-\nabla F_\e(\xi)|^2
&
=
\left|
\int_0^1
\nabla^2 F_\e(\xi +t\zeta) \zeta \, dt \right|^2
\\
&
\leq \int_0^1 |\nabla^2 F_\e(\xi+ t\zeta)\zeta |^2\, dt 
\\
&\leq
\Lambda
\int_0^1 \langle \nabla^2 F_\e(\xi+t\zeta)\zeta,\zeta\rangle \, dt
\\
&
=\Lambda \langle \nabla F_\e(\xi+\zeta)-\nabla F_\e(\xi),\zeta\rangle.
\end{align*} 
Letting $\e\to 0$ we deduce
\begin{align*}
|\nabla F(\xi+\zeta)-\nabla F(\xi)|^2
&
\leq
\Lambda \langle \nabla F(\xi+\zeta)-\nabla F(\xi),\zeta\rangle,
\end{align*}
for all $\xi\in B_{r/2}(\xi_0)$ and $\zeta\in B_{r/2}$.
This shows that $\xi_0\in V_\Lambda(\nabla F)$ and concludes the proof that $\R^2\setminus \widetilde{\mathcal S}(\nabla F)\subset\R^2\setminus \mathcal S(\nabla F)$.
\qed

\section{Example}\label{s:ex}

\begin{proof}[Proof of Theorem~\ref{t:ex}]

We start from the observation that the Lipschitz function $f\colon B_1\to\C$ given by
\begin{align*}
f(re^{i\theta})=\frac 23 r ie^{2i\theta},
\end{align*}
is not $C^1$ at the origin and satisfies
\begin{align*}
f_z =ie^{i\theta},\quad f_{\bar z} =-\frac 13 i e^{3i\theta} = \frac 13 (f_z)^3.
\end{align*}
We claim that there exists a smooth function $H\colon \C\to\C$ with compact support and such that
\begin{align}\label{eq:Hex}
H(z)=\frac{z^3}{3}\quad\forall z\in\mathbb S^1,
\qquad\text{ and }\|\nabla H(z)\|<1\quad\forall z\in\C\setminus \mathbb S^1.
\end{align}
Here $\|\cdot\|$ denotes the operator norm.
Since $\mathbb S^1$ contains no segment, this implies that $H$ is strictly 1-Lipschitz.
And since $|f_z|=1$, the function $f$ is a solution of $f_{\bar z}=H(f_z)$.

The construction of $H$ can be achieved e.g. by setting
$H(re^{i\theta})=g(r)e^{3i\theta}/3$
%\begin{align*}
%H(re^{i\theta})=g(r)\frac{e^{3i\theta}}{3},
%\end{align*}
with $g\colon (0,\infty)\to \R$ smooth, compactly supported and satisfying
\begin{align*}
g(1)=1,\quad \text{and } |g(r)|<r,\; |g'(r)|<3,\quad\forall r\neq 1.
\end{align*}
Note that
this
 forces $g'(1)=1$,
 since the quotient
  $(r-g(r))/(r-1)$ is positive for $r>1$, negative for $r<1$,
  and tends to $1-g'(1)$ as $r\to 1$.

The differential of $H$ is given by
\begin{align*}
\nabla H(re^{i\theta})
&
=\partial_r H \otimes e^{i\theta} +\frac 1r \partial_\theta H \otimes ie^{i\theta}
\\
&
=\frac{g'(r)}{3} e^{3i\theta}\otimes e^{i\theta} 
+ \frac{g(r)}{r}ie^{3i\theta}\otimes ie^{i\theta},
\end{align*}
which implies
\begin{align*}
\|\nabla H(re^{i\theta})\| =\max \left( \frac{|g'(r)|}{3},\frac{|g(r)|}{r}\right) < 1,
\quad
\forall r\neq 1,
\end{align*}
and ensures that \eqref{eq:Hex} is satisfied.

Since $H$ is compactly supported, 
Proposition~\ref{p:H} 
provides a continuous strictly monotone vector field $G$ such that 
$u=\frac 12\re f$ 
solves $\dv G(\nabla u)=0$ in $B_1$. 
This function $u(re^{i\theta})=-(r/3)\sin(2\theta)$ is Lipschitz but not $C^1$.

The sets $\Gamma_\pm$ associated to $H$ are easily calculated. 
We have $\Gamma_\pm\subset \mathbb S^1$,
since outside $\mathbb S^1$ the function $H$ is smooth with $\|\nabla H\|<1$, and noting that
\begin{align*}
\lim_{t\to 0}
L_H(e^{i\theta},e^{i\theta}(e^{it}-1))
&
=-e^{4i\theta},
\end{align*}
we see that $e^{i\theta}\in\Gamma_\pm$ for all $\theta\notin \frac\pi 4 \mathbb Z$. 
As this sets are closed we infer $\Gamma_+=\Gamma_-=\mathbb S^1$.
This implies that $\mathcal D(G)=\mathcal S(G)=F(\mathbb S^1)$, with $F(\xi)=(H(\xi)+\bar \xi)/2$ as in Proposition~\ref{p:H}.

Finally we show that 
\begin{align}\label{eq:extildeS}
\widetilde{\mathcal S}(G)=F(\lbrace\pm 1,\pm i\rbrace).
\end{align}
First note that $\widetilde{\mathcal S}(G)\subset \mathcal S(G)= F(\mathbb S^1)$, so it suffices to consider the behavior of $G$ around points $F(e^{i\theta})$.
We have
\begin{align*}
2\nabla F(e^{i\theta})
&
= 
e^{-i\theta}\otimes e^{i\theta} -
 ie^{-i\theta}\otimes ie^{i\theta} 
 + \nabla H(e^{i\theta}) \\
&=\left(e^{-i\theta} +\frac 13 e^{3i\theta}\right)\otimes e^{i\theta} 
+ \left(ie^{3i\theta}-ie^{-i\theta}\right)\otimes ie^{i\theta},
\end{align*}
hence the matrix of $\nabla F(e^{i\theta})$ in the orthonormal basis $(e^{i\theta},ie^{i\theta})$ is given by
\begin{align*}
[\nabla F(e^{i\theta})]
&
=
\left(
\begin{array}{cc}
\frac 23\cos(2\theta)  & 
-\sin(2\theta) 
\\
-\frac 13 \sin(2\theta) &
0
\end{array}
\right).
\end{align*}
In particular we see that
$\det \nabla F(e^{i\theta})=-(1/3)\sin^2(2\theta)$.
If $\theta\notin \frac\pi 2\mathbb Z$, the inverse function theorem
ensures that $F$ is a
local $C^1$ diffeomorphism in a neighborhood of $e^{i\theta}$,
and so $G=-F_*\circ F^{-1}$
is $C^1$ in a neighborhood of $F(e^{i\theta})$. This implies already that $\widetilde{\mathcal S}(G)\subset F(\lbrace \pm 1,\pm i\rbrace)$. 
To 
prove
 \eqref{eq:extildeS} we calculate $DG(F(e^{i\theta}))$ for $\theta\notin \frac\pi 2\mathbb Z$.

Differentating the identity $G(F(z))=-iF_*(z)=(\bar z -H(z))/2$, we obtain
\begin{align*}
2\nabla G(F(e^{i\theta}))\nabla F(e^{i\theta})
&
=
e^{-i\theta}\otimes e^{i\theta} -
 ie^{-i\theta}\otimes ie^{i\theta} 
- \nabla H(e^{i\theta}) \\
&=\left(e^{-i\theta} -\frac 13 e^{3i\theta}\right)\otimes e^{i\theta} 
- \left(ie^{3i\theta}+ie^{-i\theta}\right)\otimes ie^{i\theta},
\end{align*}
or, in the orthonormal basis $(e^{i\theta},ie^{i\theta})$,
\begin{align*}
[\nabla G(F(e^{i\theta}))
\nabla F(e^{i\theta})]
&
=
\left(
\begin{array}{cc}
\frac 13\cos(2\theta)  & 
0
\\
-\frac 23 \sin(2\theta) &
-\cos(2\theta)
\end{array}
\right).
\end{align*}
Using the above expression of $\nabla F(e^{i\theta})$ we deduce
\begin{align*}
[\nabla G(F(e^{i\theta}))]
%&
%=-\frac{3}{\sin^2(2\theta)}
%\left(
%\begin{array}{cc}
%\frac 13\cos(2\theta)  & 
%0
%\\
%-\frac 23 \sin(2\theta) &
%-\cos(2\theta)
%\end{array}
%\right)
%\left(
%\begin{array}{cc}
%0  & 
%\sin(2\theta) 
%\\
%\frac 13 \sin(2\theta) &
%\frac 23\cos(2\theta)
%\end{array}
%\right)
%\\
&
=
\frac{1}{\sin^2(2\theta)}
\left(
\begin{array}{cc}
0 & 
- \cos(2\theta)\sin(2\theta)
\\
 \cos(2\theta)\sin(2\theta)
 &
2
\end{array}
\right),
\end{align*}
so that the symmetric part is given by
\begin{align*}
\nabla^s G(F(e^{i\theta}))
=\frac{2}{\sin^2(2\theta)}
ie^{i\theta}\otimes ie^{i\theta}\, ,
\qquad\forall \theta\notin\frac\pi 2\mathbb Z,
\end{align*}
and this implies that $F(\pm 1),F(\pm i) \in \widetilde{\mathcal S}(G)$, 
since otherwise $\nabla^s G$ would be bounded near these points. 
Indeed, if
 $\xi_0\notin\widetilde{\mathcal S}(G)$
%and 
then there exist $\Lambda , r> 0$ such that 
$\limsup_{|\zeta|\to 0} |\zeta|^{-2}\langle D^\zeta G(\xi),\zeta\rangle \leq\Lambda$ for all $\xi\in B_r(\xi_0)$, 
and if $G$ is differentiable in $B_\rho(\xi_0)\setminus\lbrace \xi_0\rbrace$ for some $0<\rho\leq r$, this implies $\nabla^s G\leq\Lambda$ in $B_\rho(\xi_0)$.
\end{proof}

\begin{appendices}

\section{Modification and approximation lemmas}

In this appendix we prove various technical results needed 
for the approximation argument in \S~\ref{s:proof}.
First we show how to modify
 $G$ at infinity so that we can assume the set $\mathcal{D}(G) \cap \mathcal{S}(G)$ finite in the whole plane, along with some other technical conditions.

\begin{lemma}\label{l:modif}
Let $G \colon \mathbb{R}^2 \to \mathbb{R}^2 $ a continuous strictly monotone vector field, 
and $M>0$.
%such that $\mathcal{D}(G) \cap \mathcal{S}(G) \cap \overline{B_M}$ is finite for some $M>0$.
Then there exists $\widetilde G\colon\R^2\to\R^2$ a continuous strictly monotone vector field  equal to $G$ in $\overline B_M$ and smooth outside $B_{4M}$, such that
\begin{align*}
\mathcal{D}(\widetilde G)\cap\mathcal S(\widetilde G)\subset \mathcal{D}(G) \cap \mathcal{S}(G) \cap \overline{B_M},
\end{align*}
and
\begin{align*}
&
c\leq \nabla^s \widetilde G(\xi) \leq |\nabla \widetilde G(\xi)|\leq 4c
\quad \forall \xi \in \R^2\setminus B_{4M},
\\
&
|\widetilde G(\xi)|\leq L(1+|\xi|)\quad\forall \xi\in\R^2,
\end{align*}
for some constants $L,c>0$ depending on $M$ and $\|G\|_{L^\infty(B_{4M})}$.
\end{lemma}

\begin{proof}
Fix a smooth cut-off function $\eta$ such that
\begin{align*}
\mathbf 1_{B_{2M}}\leq \eta\leq \mathbf 1_{B_{4M}}
\text{ and }
|\nabla \eta|\leq \frac{1}{M} \mathbf 1_{B_{4M}\setminus B_{2M}},
\end{align*}
and a 
convex function $F(x)=c\mathbf 1_{|x|\geq M}(|x|-M)^2$, with $c>0$ to be chosen later on.
Note that
\begin{align*}
2c \frac{|x|-M}{|x|}
\mathbf 1_{|x|\geq M}  \leq \nabla^2 F(x) \leq 2c \mathbf 1_{|x|\geq M},
\end{align*}
in the sense of distributions. Define
\begin{align*}
\widetilde G =\eta G +\nabla F.
\end{align*}
The function $\widetilde G$ is continuous and equal to $G$ in
$\overline B_M$. Outside $B_{4M}$, it is smooth equal to $\nabla F$
and  $c\leq \nabla^s \widetilde G \leq |\nabla\widetilde G|\leq 4c$.
And for all $\xi\in\R^2$ we have
\begin{align*}
|\widetilde G(\xi)| 
&
\leq \|G\|_{L^\infty(B_{4M})} + 2c (|\xi| +M)
\leq L (1+|\xi|),
\end{align*}
with $L=2c +2cM +\|G\|_{L^\infty(B_{4M})}$.
It remains to check that $\widetilde G$ is strictly monotone and that
$\mathcal{D}(\widetilde G)\cap\mathcal S(\widetilde G)\subset \mathcal{D}(G) \cap \mathcal{S}(G) \cap \overline{B_M}$.

The distributional symmetric gradient of $\widetilde G$ is given by 
\begin{align*}
\nabla^s \widetilde G =\eta \nabla^s G +\nabla^2 F + \nabla\eta \odot G,
\end{align*}
where $a\odot b =(a\otimes b)^s$ is the matrix with entries $(a_ib_j+a_jb_i)/2$.
From the properties of $\eta$ and $F$ we have
\begin{align*}
\frac 12\nabla^2 F  +\nabla\eta\odot G 
&
\geq 
c \frac{|x|-M}{|x|}
\mathbf 1_{|x|\geq M} - 
\frac
{\|G\|_{L^\infty(B_{4M})}}
{M}
\mathbf 1_{2M\leq |x|\leq 4M}
\\
&
\geq
\left(\frac{c}{2}-\frac{\|G\|_{L^\infty(B_{4M})}}{M}\right)\mathbf 1_{|x|\geq 2M}
\geq 0,
\end{align*}
provided we chose $c=2\|G\|_{L^\infty(B_{4M})}/M$.
Then we deduce
\begin{align}\label{eq:symgradtildeG}
\nabla^s \widetilde G \geq \eta \nabla^s G +\frac 12 \nabla^2 F \geq\frac 12 \nabla^2 F.
\end{align}
In particular, the distributional symmetric gradient $\nabla^s\widetilde G$ is nonnegative, 
so  $\widetilde G$ is monotone.

If $\widetilde G$ is not strictly monotone then there is a nontrivial segment $[\xi,\xi+\zeta]$ along which $\langle \zeta,\widetilde G\rangle$ is constant.
This segment must intersect either $B_{M}$ or $\R^2\setminus B_M$.
If $[\xi,\xi+\zeta]$ intersects $B_M$, then this is impossible because $\widetilde G=G$ in $B_M$ and $G$ is strictly monotone.
If $[\xi,\xi+\zeta]$ intersects $\R^2\setminus B_M$, then this is also impossible because there we have
$\nabla^s \widetilde G\geq\frac 12 \nabla^2 F >0$.
We infer that $\widetilde G$ is strictly monotone.

From the inequalities \eqref{eq:symgradtildeG} we have
 $\mathcal D(\widetilde G)\subset \mathcal D(G)\cap \overline B_M$.
 %This implies $\mathcal D(\widetilde G)\cap\mathcal S(\widetilde G)\subset 
% \overline B_M\cap \mathcal D(G)\cap\mathcal S(\widetilde G)$. 
 To conclude, it suffices to show that $\mathcal S(\widetilde G)\cap \overline B_M\subset\mathcal S(G)\cup\mathcal D(G)$,
 which will imply that 
 $\mathcal D(\widetilde G)\cap\mathcal S(\widetilde G)\subset 
 \overline B_M\cap \mathcal D(G)\cap\mathcal S(G)$.
 
 If $\xi\in \overline B_M \setminus \mathcal S(G)\cup\mathcal D(G)$, there exist $\Lambda,\lambda>0$ such that $\xi\in O_{4\lambda}(G)\cap V_{\Lambda/4}(G)$.
% \begin{align*}
% B_r(\xi)\subset O_\lambda(G)\cap V_\Lambda(G).
% \end{align*}
This implies the existence of a small $r\in (0,M)$ such that, for all $\zeta\in B_r$,
\begin{align*}
&
\langle G(\xi+\zeta)-G(\xi),\zeta \rangle \geq 2\lambda |\zeta|^2,
\\
\text{and}
\quad
&
\langle G(\xi+\zeta)-G(\xi),\zeta\rangle \geq \frac 2{\Lambda} |G(\xi+\zeta)-G(\xi)|^2.
\end{align*}
Setting $\alpha=\min(\lambda,1/\Lambda)>0$, we deduce
 \begin{align*}
 \langle G(\xi+\zeta)-G(\xi),\zeta\rangle 
 \geq \alpha |\zeta|^2
+ 
 \alpha |G(\xi+\zeta)-G(\xi)|^2,
 \quad\forall \zeta\in B_r.
 \end{align*}
 Since $\widetilde G=G+\nabla F$ in $B_{2M}$, we infer, for any $\beta\in (0,\alpha/2)$,
 \begin{align*}
 &
 \langle \widetilde G(\xi+\zeta)-\widetilde G(\xi),\zeta\rangle 
 \\
 &
 \geq 
 \alpha |\zeta|^2
+ 
 \alpha |G(\xi+\zeta)-G(\xi)|^2
 %+\langle \nabla F(\xi+\zeta)-\nabla F(\xi),\zeta\rangle
 \\
 &
 \geq 
 \alpha |\zeta|^2
+ 
\beta |\widetilde G(\xi+\zeta)-\widetilde G(\xi)|^2
%\\
%&\quad
%+\langle \nabla F(\xi+\zeta)-\nabla F(\xi),\zeta\rangle
-2\beta |\nabla F(\xi+\zeta)-\nabla F(\xi)|^2
\\
 &
\geq 
( \alpha -4c\beta) |\zeta|^2
+ 
\beta |\widetilde G(\xi+\zeta)-\widetilde G(\xi)|^2.
 \end{align*}
 In the last inequality we have used that $\nabla F$ is $2c$-Lispchitz. Choosing $\beta\leq \alpha/4c$, we deduce that $\xi\notin \mathcal S(\widetilde G)$.
This shows that $\mathcal S(\widetilde G)\cap \overline B_M\subset\mathcal S(G)\cup\mathcal D(G)$ and concludes the proof.
\end{proof}

%Next, we have the following Theorem. 

Next,
we establish that $G$ can be approximated by smooth strongly monotone vector fields, 
with control on the modulus of monotony and on the sets $O_\lambda$, $V_\Lambda$.

\begin{lemma}
\label{l:approxG}
Let $G \colon \mathbb{R}^2 \to \mathbb{R}^2 $ a continuous strictly monotone vector field. 
Assume that there exist $M,L\geq 1$, $c >0$ such that $G$ is smooth in $\R^2\setminus B_{4M}$
and
\begin{align*}
&
c\leq \nabla^s  G(\xi) \leq |\nabla  G(\xi)|\leq 4c
\quad \forall \xi \in \R^2\setminus B_{4M},
\\
&
| G(\xi)|\leq L(1+|\xi|)\quad\forall \xi\in\R^2.
\end{align*}
Then there exists a sequence $G_\e$ of smooth and 
strongly monotone \eqref{eq:strongmonot} vector fields such that
$G_\e\to G$ locally uniformly as $\e\to 0$, and
\begin{align*}
&
\nabla^s G_\e(\xi) \geq c
\qquad \forall \xi \in \R^2\setminus B_{5M},
\\
&
| G_\e(\xi)|\leq 2L(1+|\xi|)\qquad\forall \xi\in\R^2\, 
,
\\
&
 \omega_{G_{\e}} \geq \omega_G, 
 \\
&
B_{2\e}(\xi)\subset O_\lambda(G)\;\Rightarrow\; \xi\in O_\lambda(G_\e)\, ,
%\quad\forall \lambda>0,
\\
& 
B_{2\e}(\xi)\subset V_{\Lambda}(G)\;\Rightarrow\; \xi\in V_{\Lambda+\e}(G_\e)\, ,
%\quad\forall \Lambda >0.
\end{align*}
for all $\e\in (0,1)$.
\end{lemma}

\begin{proof}
We fix a smooth kernel $\rho \in C_c^{\infty}(B_1)$, such that $\rho\geq 0$ and $\int_{B_1} \rho = 1$,
define  $\rho_{\e}(\xi)= \e^{-2} \rho(\xi/\e)$,
and
\begin{align*}
G_\e(\xi) =G * \rho_\e (\xi) +\e\, \xi.
\end{align*}
Then $G_\e$ is smooth and converges locally uniformly to $G$.

It is
globally Lipschitz because $|\nabla G_\e|\leq \e^{-1}\|G\|_{L^\infty(B_{6M})}$ on $B_{5M}$ and $|\nabla G_\e|\leq 4c$ outside $B_{5M}$.
Global Lipschitzness combined with the inequality
\begin{align*}
\nabla^ s G _\e = \nabla^s G *\rho_\e +\e \, I \geq \e,
\end{align*}
implies that $G_\e$ is strongly monotone.

Outside $B_{5M}$ we have $\nabla^s G_\e\geq \nabla^s G *\rho_\e \geq c$.
And for all $\xi\in\R^2$ we have $|G_\e(\xi)|\leq L (1+|\xi|+\e) + \e\, |\xi|\leq 2L(1+|\xi|)$.

In the rest of the proof we use the notation $D^\zeta$ for the finite difference operator \begin{align*}
D^\zeta G(\xi)=G(\xi+\zeta)-G(\xi).
\end{align*}
For any $\xi,\zeta\in\R^2$ we have
\begin{align*}
&
\langle  D^\zeta G_\e(\xi),\zeta\rangle
=\int_{B_1} \langle D^\zeta G(\xi+\e\eta),\zeta\rangle \rho(\eta)d\eta +\e\,|\zeta|^2
\geq \omega_G(|\zeta|),
\end{align*}
so that $\omega_{G_\e}\geq \omega_G$.

Let $\lambda>0$ and assume that $B_{2\e}(\xi_0)\subset O_\lambda(G)$.
Let $\xi\in B_\e(\xi_0)$, so that $B_\e(\xi)\subset O_\lambda(G)$.
By definition \eqref{eq:Olambda} of $O_\lambda$,
 for all $\eta\in B_1$ there exists
 $\varphi(\eta,r)$  such that $0 < \varphi \leq \lambda$, $\varphi(\eta,r)\to 0$ as $r\to 0$ and
\begin{align*}
\langle D^\zeta G(\xi+\e\eta),\zeta\rangle
\geq (\lambda - \varphi(\eta,|\zeta|) ) |\zeta|^2.
\end{align*}
Then we have
\begin{align*}
\langle D^\zeta G_\e(\xi),\zeta\rangle
&
=\int_{B_1} \langle D^\zeta G(\xi+\e\eta),\zeta\rangle \rho(\eta)d\eta +\e\,|\zeta|^2
\\
&\geq \left(\lambda -\psi(|\zeta|)\right) |\zeta|^2,
\qquad\psi(r)=\int_{B_1}\varphi(\eta,r)\,\rho(\eta)\, d\eta,
\end{align*}
and $\psi(r)\to 0$ as $r\to 0$ by dominated convergence, so
\begin{align*}
\liminf_{|\zeta|\to 0}
\frac{ \langle D^\zeta G_\e(\xi),\zeta\rangle}{|\zeta|^2} \geq \lambda\qquad\forall \xi\in B_\e(\xi_0),
\end{align*}
and we deduce that
 $\xi_0\in O_\lambda(G_\e)$.

Now let $\Lambda >0$
and
assume that
$B_{2\e}(\xi_0)\subset V_\Lambda(G)$.
In order to show that $\xi_0\in V_{\Lambda +\e}(G_\e)$ we argue
slightly differently than for $O_\lambda$.

First we observe that $G$ is Lipschitz in $B_{2\e}(\xi_0)$ (see Lemma~\ref{l:VLambda} below), 
 hence differentiable almost everywhere.
 Then
 the inclusion $B_{2\e}(\xi_0)\subset V_\Lambda(G)$ implies
\begin{align*}
|\nabla G(\xi)\zeta|^2\leq \Lambda \langle \nabla G(\xi)\zeta,\zeta\rangle\qquad\text{for a.e. }\xi\in B_{2\e}(\xi_0)\text{ and all }\zeta\in\R^2.
\end{align*}
Thus we find, for $\xi\in B_\e(\xi_0)$ and $\zeta\in\R^2$,
\begin{align*}
|\nabla G_\e(\xi)\zeta|^2
=
&
\left|\int_{B_1}\nabla G(\xi+\e\eta)\zeta \, \rho(\eta)\,d\eta +\e\zeta
\right|^2
\\
&
\leq \left(1+\frac\e\Lambda \right)
 \int_{B_1}|\nabla G(\xi+\e\eta)\zeta|^2\rho(\eta)\, d\eta 
 +(\e+\Lambda ) \e|\zeta|^2
\\
&
\leq 
(\Lambda +\e)
 \int_{B_1}\langle \nabla G(\xi+\e\eta)\zeta,\zeta\rangle\,\rho(\eta)d\eta
+(\e+\Lambda ) \e|\zeta|^2
 \\
 &
=(\Lambda+\e) \langle \nabla G_\e(\xi)\zeta,\zeta\rangle\, .
\end{align*}
This implies that $\xi_0\in V_{\Lambda +\e}(G_\e)$.
\end{proof}

In the proof of Lemma~\ref{l:approxG} we used the following 
elementary property of the set $V_\Lambda$.

\begin{lemma}\label{l:VLambda}
If $G \colon \mathbb{R}^2 \to \mathbb{R}^2 $ is a continuous strictly monotone vector field and 
$B_r(\xi_0)\subset V_\Lambda(G)$ for some $\xi\in\R^2$
and
 $\Lambda,r>0$, then $G$ is $\Lambda$-Lipschitz in $B_r(\xi_0)$.
\end{lemma}
\begin{proof}[Proof of Lemma~\ref{l:VLambda}]
By definition \eqref{eq:VLambda} of $V_\Lambda$, for any fixed $\xi\in B_{r}(\xi_0)$ and $\delta>0$, if $|\zeta|$ is small enough we have
\begin{align*}
|D^\zeta G(\xi)|^2 
&
\leq (1+\delta)\Lambda
\langle D^\zeta G(\xi),\zeta\rangle
%\\
%&
\leq 
\frac 12 |D^\zeta G(\xi)|^2 + \frac 12 (1+\delta)^2\Lambda^2 |\zeta|^2,
\end{align*}
so that, letting $|\zeta|\to 0$ and then $\delta\to 0$ we deduce
\begin{align*}
\limsup_{\zeta\to 0} \frac{|D^\zeta G(\xi)|}{|\zeta|}\leq \Lambda\quad\forall \xi\in B_{r}(\xi_0)\,
.
\end{align*}
This infinitesimal Lipschitz property implies that $G$ is $\Lambda$-Lipschitz in the convex set $B_{r}(\xi_0)$.
Indeed, for $[\xi,\xi +\zeta]\subset B_{r}(\xi_0)$ and $\delta>0$,
by compactness and infinitesimal Lipschitzness we can find $0=t_0<t_1<\cdots <t_N=1$ such that $|G(\xi +t_{j+1}\zeta)-G(\xi+t_j\zeta)|\leq (1+\delta)\Lambda (t_{j+1}-t_j)|\zeta|$, and concatenating these inequalities gives $|D^\zeta G(\xi)|\leq (1+\delta)\Lambda |\zeta|$.
\end{proof}

Finally, we check that solutions of the equation given by the smooth approximating vector fields $G_\e$ are locally uniformly Lipschitz, thanks to the results of \cite{EMT}.

\begin{lemma}\label{l:approxu}
Let $G,G_\e \colon \R^2 \to \R^2 $ be as in Lemma~\ref{l:approxG}, 
and $u$
 a solution of $\dv G(\nabla u)=0$ in $B_1$ with $|\nabla u|\leq M$. 
 For $\e\in (0,1)$, let $u_{\e}$ be the unique smooth solution of the boudary value problem 
\begin{align*}
\dv G_{\e}(\nabla u_{\e} )=0 
\quad \text{in } B_1,
\qquad
u_{\varepsilon}=u  
\quad \text{in } \partial B_1.
\end{align*} 
Then we have
\begin{align*}
\sup_{\e\in (0,1)}\|\nabla u_\e\|_{L^\infty(K)} <\infty\qquad\text{ for all compact }K\subset B_1,
\end{align*}
and $u_{\e} \to u $ locally uniformly in $B_1$, and strongly in $H^1(B_1)$.
\end{lemma}

\begin{proof}[Proof of Lemma~\ref{l:approxu}] 
The existence of a unique solution $u_\e\in H^1(B_1)$ follows from the strict monotony of $G_\e$ and its behavior at infinity, see e.g.
 \cite[\S~9.1]{EvansPDE}.
 Moreover, this solution $u_\e$ is Lipschitz thanks to \cite[Theorem~4.1]{EMT} (applied with $g(t)=t^2$), and therefore smooth since $G_\e$ is smooth.
On each compact $K\subset B_1$, the Lipschitz constant of $u_\e$ provided by \cite[Theorem~4.1]{EMT} depends on the constants $c,L,M$ such that
\begin{align}\label{eq:hypGcLM}
\nabla^s G_\e(\xi) \geq c
\quad \forall \xi \in \R^2\setminus B_{5M},
\quad
| G_\e(\xi)|\leq 2L(1+|\xi|)\quad\forall \xi\in\R^2\, ,
\end{align}
and on the $L^2$ norm of $\nabla u_\e$ in $B_1$.
Therefore, the locally uniform Lipschitz bound will follow from  
\begin{align*}
\sup_{\e\in (0,1)}\int_{B_1}|\nabla u_\e|^2\, dx<\infty.
\end{align*}
From \eqref{eq:hypGcLM} we infer the existence of $D>0$ depending on $M$ and $L$ such that
\[
\langle G_{\e}(\xi)-G_{\e}(\zeta),\xi-\zeta \rangle \geq c \vert \xi-\zeta \vert^2 -D                 
\qquad\forall \xi,\zeta\in\R^2.
\] 
Testing the equation $\dv (G_\e(\nabla u_\e)-G_\e(\nabla u))=\dv(G(\nabla u)-G_\e(\nabla u))$ against $u_\e-u$, we infer
\begin{align*}
c\int_{B_1}|\nabla u_\e-\nabla u|^2\, dx 
&
\leq D +\int_{B_1} \|G_\e-G\|_{L^\infty(B_M)} |\nabla u-\nabla u_\e|\, dx
\\
&
\leq 
D +\frac{\pi M^2}{2c}\|G_\e-G\|_{L^\infty(B_M)}^2 +\frac c 2 \int_{B_1}|\nabla u_\e-\nabla u|^2\, dx,
\end{align*}
which implies the uniform boundedness of $\int_{B_1}|\nabla u_\e-\nabla u|^2\, dx$, 
and therefore of $\int_{B_1}|\nabla u_\e|^2\, dx$ since $u$ is $M$-Lipschitz.

We may extract a subsequence $\e_k\to 0$ such that $u_{\e_k}$ converges locally uniformly in $B_1$ and weakly in $H^1(B_1)$.
Passing to the limit in the identity
\begin{align*}
\int_{B_1}\langle G_\e(\nabla u_\e)-G(\nabla u),\nabla u_\e-\nabla u\rangle\, dx =0,
\end{align*}
using the locally uniform convergence $G_\e\to G$ and the strict monotony of $G$,
we obtain that any Young measure generated by a subsequence of $(\nabla u_\e)$ is concentrated at $\nabla u$, see e.g.  \cite[Lemma~2.8 \& Remark~2.9]{Lledos} for details.
We conclude that $u_\e\to u$ locally uniformly, and strongly in $H^1(B_1)$.
\end{proof}

\section{Degenerate linear Beltrami equations}\label{a:linbeltrami}

In this appendix we prove the assertion claimed in Remark~\ref{r:conjug}.

\begin{proposition}\label{p:linbeltrami}
Let $\mu, \nu \in \mathbb{C}$ such that $ \vert \mu \vert + \vert \nu \vert =1 $.
 Then, for any open set $\Omega\subset \C$,
 the implication
\[
f_{\bar z} = \mu f_z + \nu \overline{f_z} 
\quad\Rightarrow
\quad \re f \text{ or } \im f \text{ is constant} ,\]
is true for all differentiable $f\colon\Omega\to\C$
if and only if $\mu=0$ and $\nu=\pm 1$.
\end{proposition}

\begin{proof}
With $f=u+iv$, the equation $f_{\bar z} = \mu f_z + \nu \overline{f_z} $ is equivalent to
\begin{equation*}
\begin{cases}
(1- \mu - \nu ) \partial_x u - (1+ \mu + \nu )\partial_y v =0  \\
(1+\mu-\nu)\partial_y u + (1-\mu + \nu)\partial_x v =0.
\end{cases}
\end{equation*}
If $(\mu,\nu)=(0,-1)$ or $(0,+1)$, one directly deduces $\partial_xu=\partial_y u=0$ or $\partial_x v =\partial_y v=0$, so that either $u$ or $v$ is constant.

Conversely, assume that $(\mu,\nu) \notin \lbrace (0,\pm 1)\rbrace$, and let us construct 
an affine solution $f$ such that both $u$ and $v$ are not constant. 
We rewrite the system as
\begin{equation*}
\begin{cases}
a\partial_x u = b \partial_y v , \\
c \partial_y u = d \partial_x v,
\end{cases}
\quad (a,c)\neq (0,0)\text{ and }(b,d) \neq (0,0).
\end{equation*} 
If $a \neq 0$ and $ b \neq 0$, we set $ u(x,y)= \frac{b}{a} x $ and $ v(x,y) = y $. Then $f = u+iv$ satisfies the equation and both $u$ and $v$ are not constant. 

If $a \neq 0$ and $b= 0$, then $d \neq 0$. We have two cases in this situation depending on the value of $c$. If $c=0$, we can take $u(x,y)=v(x,y)=y$ as a solution. If $ c \neq 0$ then, $u(x,y) = \frac{d}{c} y $ and $v(x,y)=x$ is a solution. 

The cases where 
$a=0$ 
  can be dealt with similarly. 
\end{proof}

\end{appendices}

\bibliographystyle{acm}
\bibliography{ref}

\end{document}